\documentclass[letter,12pt,fleqn]{amsart}
\usepackage[english]{babel}
\usepackage[utf8x]{inputenc}
\usepackage{amsmath}
\usepackage{amssymb}
\usepackage{amsthm}
\usepackage{fourier}

\usepackage{mathtools}
\usepackage{graphicx}
\usepackage{color}
\usepackage{algorithm}
\usepackage{algpseudocode}
\usepackage{algorithmicx}
\usepackage{caption}
\usepackage{subcaption}

\usepackage{indentfirst}

\usepackage{verbatim}
\usepackage{epstopdf}
\usepackage{epsfig}

\algblockdefx[foruntil]{ForUntilBegin}{ForUntilEnd}
	[2]{\textbf{for} #1 \textbf{do} }
	[1]{\textbf{until }#1}

\usepackage[numbers]{natbib}
\usepackage[margin=1in]{geometry}

\newtheorem{theorem}{Theorem}
\newtheorem{corollary}{Corollary}
\newtheorem{lemma}{Lemma}[section]

\usepackage{fancyhdr}
\usepackage{letltxmacro}
\LetLtxMacro{\existsold}{\exists}
\renewcommand{\exists}{\existsold \hspace{.2em} }
\LetLtxMacro{\forallold}{\forall}
\renewcommand{\forall}{\: \forallold \: }
\LetLtxMacro{\intnolim}{\int}
\renewcommand{\int}{\intnolim\limits }
\renewcommand{\tilde}{\widetilde}

\newcommand{\norm}[1]{\left\lVert #1 \right\rVert }

\renewcommand{\Re}{\text{Re} }
\renewcommand{\Im}{\text{Im} }

\renewcommand{\epsilon}{\varepsilon }
\DeclareMathOperator{\tr}{tr}
\DeclareMathOperator{\col}{Col}
\DeclareMathOperator{\nul}{Null}

\newcommand{\R}{\mathbb{R}}
\DeclareMathOperator*{\argmin}{arg\,min}
\newcommand{\tnorm}[1]{\VERT #1 \VERT}

\DeclareMathOperator{\spanop}{span}

\title{Orbital minimization method with $\ell^1$ regularization}

\author{Jianfeng Lu}
\address{Department of Mathematics, Department of
  Physics, and Department of Chemistry, Duke University, Box 90320, Durham NC 27708, USA}
\email{jianfeng@math.duke.edu}

\author{Kyle Thicke}
\address{Department of Mathematics, Duke University, Box 90320, Durham NC 27708, USA}
\email{kyle.thicke@duke.edu}

\date{\today} \thanks{This work is partially supported by the National
  Science Foundation under grants DMS-1454939 and ACI-1450280. We
  thank Fabiano Corsetti, Haizhao Yang and Wotao Yin for helpful
  discussions.}

\begin{document}

\begin{abstract}
  We consider a modification of the OMM energy functional which
  contains an $\ell^1$ penalty term in order to find a sparse
  representation of the low-lying eigenspace of self-adjoint operators.
  We analyze the local minima of the modified functional as well as
  the convergence of the modified functional to the original
  functional.  Algorithms combining soft thresholding with gradient
  descent are proposed for minimizing this new functional.  Numerical
  tests validate our approach.  As an added bonus, we also prove the
  unanticipated and remarkable property that every local minimum the
  OMM functional without the $\ell^1$ term is also a global minimum.
\end{abstract}

\maketitle

\section{Introduction}

This paper considers solving for a sparse representation of the low-lying
eigenspace of self-adjoint operators. Given a Hermitian matrix $H$,
the low-lying eigenspace is defined as the linear combination of the
first $m$ eigenvectors of $H$. Our aim is to find a sparse
representation by solving a variational problem for the eigenspace
which at the same time favors sparsity.

Such sparse representations have applications in many scenarios,
including constructing a localized numerical basis for PDE problems and
sparse PCA in statistics.  The main motivation of the current work
comes from electronic structure calculations, for which the sparse
representations, given by (approximate) linear combinations of the
low-lying eigenfunctions, are known as the Wannier functions
\cite{Kohn:1959Wannier, Wannier:1937}.

Motivated by the recent works
\cite{OzolinsLaiCaflischOsher:13,lai2014density,lai2015localized} that
exploit the $\ell^1$ penalty to enhance sparsity (see also
\cite{E:2010PNAS, Marzari:1997} for alternative localization
strategies), in this work we consider minimizing the variational
model
\begin{equation}
  E_\mu(X) = \tr[(2I-X^*X)X^*HX] + \mu \tnorm{X}_1,
\end{equation}
where $X \in \mathbb{C}^{N\times m}$ and $\tnorm{\cdot}_1$ is the entry-wise $\ell^1$ norm. Here, the energy functional without the $\ell^1$ penalty 
\begin{equation}
  E_0(X) = \tr[(2I-X^*X)X^*HX],
\end{equation} 
is used in the orbital minimization method (OMM), developed in the
context of linear scaling algorithms for electronic structure
\cite{mauri1993orbital, mauri1994electronic, Ordejon:93, Ordejon:95}
to alleviate the orthogonality constraint $X^{\ast} X = I$.  Hence, we
will refer to $E_0$ as the OMM functional and refer to $E_{\mu}$ as the OMM
functional with an $\ell^1$ penalty.

Since the orthogonality constraint is lifted, $E_{\mu}$ can be minimized with
unconstrained minimization algorithms.  This allows the algorithm to be significantly simpler than that for the trace minimization with an
$\ell^1$ penalty proposed in \cite{OzolinsLaiCaflischOsher:13} based on operator splitting \cite{Lai:2014splitting}:
\begin{equation}
  E^{\text{trace}}_{\mu}(X) = \tr [ X^{\ast} H X ] + \mu \tnorm{X}, \qquad \text{s.t.}\quad X^{\ast} X = I.  
\end{equation}
Note that it is possible to lift the orthogonality constraint by
convexification, as in \cite{lai2014density,lai2015localized}, which
leads to localized density matrix minimization. However, the density matrix $P
\in \mathbb{C}^{N\times N}$, which is the projection operator onto the
low-lying eigenspace, contains many more degrees of freedom than $X$
when $N \gg m$. Hence, it might be favorable to consider the non-convex
functional $E_{\mu}$ with fewer degrees of freedom.

For the OMM functional without the $\ell^1$ penalty, it is well known
\cite{mauri1993orbital, pfrommer1999unconstrained} that the global
minimizers of the OMM functional $E_0$ correspond to a basis for the
low-lying eigenspace of $H$, if $H$ is negative definite. Somewhat
surprisingly, it turns out that we can further show that $E_0$ has no local
minima in the sense that every local minimum is a global minimum, and
hence a representation of the low-lying eigenspace, as shown in
Theorem~\ref{thm:localglobalminima}. In particular, when minimizing
$E_0$ using e.g., the conjugate gradient method, we will not be trapped at
local minima. 

With the $\ell^1$ term, the minimizer of $E_{\mu}$ no longer
corresponds to the exact low-lying eigenspace; however, we show that
it gives a good approximation as $\mu \to 0$. We will further analyze
the approximate orthogonality of the minimizer and approximation to
the density matrix.

The OMM functional has been used to develop linear scaling electronic
structure algorithms (so that computational cost scales linearly with
the number of electrons).  See e.g., the review paper
\cite{Goedecker:99}. The conventional strategy is to simply truncate
the domain of the problem by minimizing $E_0$ over a subspace of
matrices in $\mathbb{C}^{N \times m}$ which have a particular
predefined support \cite{mauri1993orbital, mauri1994electronic} with
only $O(m)$ degrees of freedom. However, this truncation is known to
create many local minima in the problem, which trap the
minimization. While efforts to alleviate the local minima issue have
been undertaken \cite{kim1995total,gao2009orbital, AOMM}, it still remains
a challenge for practical applications. It turns out that
minimizing $E_{\mu}$ naturally leads to a soft thresholding
optimization algorithm, thanks to the $\ell^1$
penalty term.  In our numerical tests, we find that local minima are
not encountered in practice unless the penalty parameter $\mu$ is
large (details in Section \ref{subsec:localMinTest}).  Our method also
has the additional advantage that the support of the solution does not
have to be chosen ahead of time.

We also note that even in the context of a cubic scaling implementation,
the OMM algorithm still has an advantage over direct eigensolvers in terms
of scalability in parallel implementations, as recently demonstrated
in \cite{Corsetti:14}. The inclusion of the $\ell^1$ penalty controls
the sparsity of the solution, which allows for sparse matrix arithmetic
and will hopefully enable a natural transition between the cubic
scaling and reduced scaling algorithms. 

The remainder of the paper is structured as follows.  In Section 2, we introduce the different energy functionals we will be working with.  We perform analysis to examine the local minima of the functionals and to examine the convergence of the modified functional to the original functional.  In Section 3, we present algorithms for numerically minimizing our new functional.  Numerical tests are performed in Section 4 to validate the theoretical results from Section 2 as well as to examine the performance of the algorithms proposed in Section 3.

\section{Analysis}

\subsection{Original OMM energy functional}
\label{subsec:orgOMM}

When using Kohn--Sham density functional theory, one encounters the problem of calculating the low-lying eigenspace of a matrix $H$.  This can be rephrased as the following minimization problem.  Minimize the following functional over the set of matrices $Y\in \mathbb{C}^{N\times m}$ satisfying the orthonormality constraint $Y^*Y = I$.
\begin{equation}
E_{\perp}(Y) = \tr[Y^*HY]. \label{functionalOrth}
\end{equation}
The orthonormality constraint can be difficult to deal with, but we can relax the constraint by replacing $Y$ with $XS^{-1/2}$ where $S = X^*X$.  If $X$ has full rank, then $XS^{-1/2}$ satisfies the orthonormality constraint.  By making this substitution, we obtain the functional
\begin{equation}
E_{S^{-1}}(X) = \tr[S^{-1}X^*HX]. \label{functionalSinv}
\end{equation}
Minimizing \eqref{functionalSinv} is clearly equivalent to minimizing
\eqref{functionalOrth}, but we have relaxed the restrictions on the
set we are minimizing over.  Instead of minimizing only over the set
of matrices with orthonormal columns, we can now minimize over the set
of all full rank matrices in $\mathbb{C}^{N\times m}$.

The idea behind the OMM functional \cite{mauri1993orbital,
  mauri1994electronic} is to avoid computing the inverse of $S$, by
replacing $S^{-1}$ by an approximation obtained by a Neumann expansion
around $I$, the identity matrix.  In particular, we can replace
$S^{-1}$ with a matrix of the following form
\begin{equation}
Q_\mathcal{N} = \sum_{k=0}^\mathcal{N} (I - S)^k.
\end{equation}
Then instead of minimizing $E_{S^{-1}}$, we can consider minimizing the functional
\begin{equation}
E_{Q_\mathcal{N}}(X) = \tr[Q_\mathcal{N} X^*HX],
\end{equation}
where $\mathcal{N}$ is an odd integer (if $\mathcal{N}$ is even, the
functional is not bounded from below).  It has been shown that this
alternate functional has very nice properties \cite{mauri1993orbital,
  mauri1994electronic} which we summarize in the following theorem.
\begin{theorem}
\label{thm:minimizingSetOfE0}
Let $H \in \mathbb{C}^{N \times N}$ be Hermitian with eigenvalues $\lambda_1 \le \lambda_2 \le ... \le \lambda_N$.  Then the global minima of $E_{S^{-1}}$ is the set of matrices $X \in \mathbb{C}^{N \times m}$ whose column space is a span of eigenvectors corresponding to the $m$ most negative eigenvalues of $X$.  

Suppose further that $H$ is negative definite (we can always shift $H \to H - \eta I$ if necessary).  Let $\mathcal{N}$ be an odd positive integer.  Then, the functional $E_{Q_\mathcal{N}}$ has a global minimum, and in fact, $\min E_{Q_\mathcal{N}} = \min E_{S^{-1}}$.  Additionally, $X$ is a minimizer of $E_{Q_\mathcal{N}}$ if and only if $X$ is a minimizer of $E_{S^{-1}}$ and the columns of $X$ are orthonormal.
\end{theorem}
Note that the ``eigenvectors corresponding to the $m$ most negative
eigenvalues of $H$'' may not be unique since it could be that
$\lambda_m = \lambda_{m+1}$.  When we use this phrase, we are
referring to any possible choice of eigenvectors, when degeneracy
occurs.

The major conclusion of the theorem is that we can find exact
minimizers of $E_{S^{-1}}$ by minimizing the simpler functional
$E_{Q_\mathcal{N}}$ which does not require the computation of
$S^{-1}$.  This both reduces the computational cost and allows us to
minimize over all matrices rather than just those of full rank.  In
this paper, we will choose $\mathcal{N} = 1$.  In particular, we
consider the problem of minimizing the functional we call $E_0$,
defined by
\begin{equation}
E_0(X) = \tr\bigl[(2I - X^*X)X^*HX\bigr],
\end{equation}
where $X \in \mathbb{C}^{N \times m}$ and $H \in \mathbb{C}^{N \times N}$ is a Hermitian negative definite matrix with eigenvalues $\lambda_1 \le \lambda_2 \le ... \le \lambda_N < 0$.

Our first main result is the following theorem. The result is quite
unexpected as $E_0$ is clearly non-convex, while the theorem states
that local minima of $E_0$ are also global minima.

\begin{theorem}
\label{thm:localglobalminima}
Any local minimum of $E_0$ is also a global minimum.
\end{theorem}

Before proving the theorem, we note a very practical consequence of it.  Many numerical optimization algorithms can guarantee convergence to a critical point of the objective function.  If we use such an algorithm, then after the algorithm converges, we can perturb the solution.  Theorem \ref{thm:localglobalminima} tells us that if starting the algorithm from this perturbed point leads to convergence to the same value of $E_0$, then that point is likely a global minimum.

The proof of Theorem \ref{thm:localglobalminima} will follow from the
following lemmas.  The first two lemmas we state are elementary.  Then
the next two lemmas will provide the crucial tools needed to prove the
theorem.  Lemma~\ref{lem:charOfCPs} provides a characterization of the
critical points of $E_0$, and Lemma~\ref{lem:genPfrommerAnalysis}
allows us to analyze the behavior of $E_0$ near the critical points.
By using the analysis in Lemma~\ref{lem:genPfrommerAnalysis}, we can
examine all the critical points which have been characterized in
Lemma~\ref{lem:charOfCPs} to show that all the critical points which
are local minima are actually global minima.

\begin{lemma}
\label{lem:invariantSubspaces}
If $A \in \mathbb{C}^{n \times n}$ is Hermitian and $V$ is a $k$-dimensional invariant subspace of $A$, then there exist $k$
eigenvectors $\{\xi_i\}_{i=1}^k$ of $A$ such that $V =
\spanop\bigl(\{\xi_i\}_{i=1}^k\bigr)$.
\end{lemma}
\begin{proof}
Let $\{b_i\}_{i=1}^k$ be an orthonormal basis for $V$.  Define $B = [b_1, ..., b_k]$.  Then $A|_V$ with respect to the basis $\{b_i\}_{i=1}^k$ is $B^*AB$, which is Hermitian.  Therefore, $A|_V$ is diagonalizable with eigenvectors $\{\eta_i\}_{i=1}^k \subset \mathbb{C}^k$.  This means $B^*AB\eta_i = \lambda_i \eta_i$.  Multiplying on the left by $B$ and noting that $BB^* = I$ on $V$, we see that $\{B\eta_i\}_{i=1}^k$ are eigenvectors of $A$.  It is easy to see that $\{B\eta_i\}_{i=1}^k$ are independent and therefore form a basis for $V$.
\end{proof}

\begin{lemma}
\label{lem:invariance}
$E_0$ is invariant under right multiplication by elements of
$U(m)$.  That is, $E_0(XG) = E_0(X)$ for all $X \in
\mathbb{C}^{N \times m}$ and $G \in U(m)$.
\end{lemma}
\begin{proof}
Use basic properties of the trace.
\end{proof}

Before continuing, we note that since $X^*X$ is Hermitian, there exists $G \in U(m)$ such that $X^*X = G\Lambda G^*$ where $\Lambda$ is diagonal.  We will often combine this fact with Lemma \ref{lem:invariance} to assume without loss of generality that $X^*X$ is diagonal.  We will now prove the above assertions concerning $E_0$ by following a natural progression through the following lemmas.  The next lemma is interesting in its own right as it provides a characterization of the critical points of $E_0$.  Its proof proceeds in two steps.  First, we show that the eigenvalues of $X^*X$ are in $\{0,1\}$.  Then, we can use this fact to show that the column space of $X$, $\col(X)$, is the span of some set of eigenvectors of $H$.

\begin{lemma}
\label{lem:charOfCPs}
If $\nabla E_0(X) = 0$, then all eigenvalues of $X^*X$ are in $\{0,1\}$.  Additionally, the $\col(X)$ is the span of some set of eigenvectors of $H$.
\end{lemma}
\begin{proof}
  Note that if $G \in U(m)$, then $\nabla E_0(XG) = \nabla E_0(X)G$.
  So without loss of generality, assume that $X^*X$ is diagonal. By
  assumption we have,
\begin{align}
\nabla E_0(X) &= -2 XX^*HX + 4HX - 2HXX^*X \notag\\
  &= -2(XX^*H + HXX^* - 2H)X = 0. \label{gradientIsZero}
\end{align}
Denote $X = [x_1, x_2, ..., x_m]$, and let $T = XX^*H + HXX^* - 2H$. Note that $XX^* = \sum x_i x_i^*$.  By noting that the columns of $X$ are orthogonal (since $X^*X$ is diagonal), we have
\begin{align}
XX^*x_i &= x_i x_i^* x_i = \norm{x_i}^2 x_i, \\
x_i^*XX^* &= x_i^* x_i x_i^* = \norm{x_i}^2 x_i^*.
\end{align}
Now we use the fact that \eqref{gradientIsZero} implies $x_i^* T x_i = 0$.
\begin{align}
0 &= x_i^*(XX^*H + HXX^* - 2H)x_i \notag\\
  &= 2(\norm{x_i}^2 - 1) x_i^* H x_i.
\end{align}
$H$ is negative definite, so if $x_i \ne \vec{0}$, then $x_i^* H x_i < 0$.  This implies that either $x_i = \vec{0}$ or $\norm{x_i} = 1$.  Since the diagonal entries of $X^*X$ are $\norm{x_i}^2$, this implies that all eigenvalues of $X^*X$ are contained in the set $\{0,1\}$.

Now we may assume without loss of generality that $X^*X$ is diagonal with diagonal entries consisting only of 0s and 1s.  Equation \eqref{gradientIsZero} implies that for each column $x_i$ of $X$,
\begin{align}
\vec{0} &= (XX^*H + HXX^* - 2H)x_i \notag\\
  &= XX^*Hx_i + Hx_i - 2Hx_i \notag\\
  &= (XX^* - I)Hx_i.
\end{align}
But $\nul(XX^*-I) = \col(X)$, so $Hx_i \in \col(X)$ for all $i$.  This implies that $\col(X)$ is an invariant subspace of $H$.  So, by Lemma \ref{lem:invariantSubspaces}, $\col(X)$ is the span of some set of eigenvectors of $H$.
\end{proof}

Now we consider the perturbation of the energy of $E_0$ around a
critical point. The analysis is similar to, but more general than that
in \cite{pfrommer1999unconstrained}. The lemma will be used to prove
that all local minima of $E_0$ are global minima.

\begin{lemma}
\label{lem:genPfrommerAnalysis}
Let $\{y_i\}_{i=1}^N$ be a complete orthonormal set of eigenvectors of $H$ corresponding respectively to the (not necessarily ordered) eigenvalues $\{\lambda_i\}_{i=1}^N$ of $H$.  Let $Z = [z_1, z_2, ..., z_m]$ where $z_i$ is either $y_i$ or $\vec{0}$.  Let $X = [x_1, x_2, ..., x_m]$ where $x_i = z_i + \sum_{j=1}^N c_j^{(i)} y_j$.  Then, 
\begin{align}
E_0(X) - E_0(Z) &= \sum_{i=1}^m \sum_{k = m+1}^N |c_k^{(i)}|^2\Big(2\lambda_k - (\lambda_k + \lambda_i) \chi_i\Big) + 2 \sum_{i=1}^m \Big[(1 - \chi_i)|c_i^{(i)}|^2 - 2 (\text{Re}(c_i^{(i)}))^2 \chi_i\Big] \lambda_i \notag\\
  &\hspace{-6em}- \sum_{i=1}^m \sum_{\substack{j=1\\i \ne j}}^m \biggl[ c_i^{(j)} c_j^{(i)} \lambda_i \chi_i \chi_j + c_i^{(j)*} c_j^{(i)*} \lambda_j \chi_i \chi_j + |c_i^{(j)}|^2 \lambda_i \chi_i + |c_j^{(i)}|^2 \lambda_j \chi_j \notag \\
    & - |c_j^{(i)}|^2 \Big( 2\lambda_j - (\lambda_j + \lambda_i)\chi_i \Big)\biggr] \notag\\
  &\hspace{-6em}+ \mathcal{O}(\norm{X-Z}_F^3),
\end{align}
where $\chi_i = 0$ if $z_i = \vec{0}$ and $\chi_i = 1$ if $z_i = y_i$.
\end{lemma}
\begin{proof}
The calculation is tedious but straightforward.
\end{proof}

Using the above lemmas, we may now prove Theorem \ref{thm:localglobalminima} as promised.

\begin{proof}[Proof of Theorem \ref{thm:localglobalminima}]
We first note that the global minima of $E_0$ are the $X$ whose columns span the same space as the eigenvectors corresponding to the $m$ most negative eigenvalues (counting geometric multiplicity).  Due to the fact that eigenvalues can have more than one independent eigenvector, this $m$-dimensional ``minimal subspace'' may not be unique.

If $X_0$ is a local minimum, then Lemma \ref{lem:charOfCPs} implies that the column space of $X_0$ is spanned by some set $\{y_i\}_{i=1}^r$ of orthonormal eigenvectors of $H$.  Without loss of generality, we can assume that $X_0 = [y_1, y_2, ..., y_r, \vec{0}, ..., \vec{0}]$.  We can extend the set of $y_i$'s to a complete orthonormal set of eigenvectors of $H$, $\{y_i\}_{i=1}^N$.  Then we can apply Lemma \ref{lem:genPfrommerAnalysis} with $Z = X_0$.   

Our last step is to show that if the columns of $X_0$ are not
eigenvectors corresponding to the $m$ most negative eigenvalues of
$H$, then $X_0$ is not a local minimum.  First, if $X_0$ contains a
column which is $\vec{0}$, then it is easy to see that there must
exist $C_1>0$ such that for $0<|c_m^{(m)}| < C_1$, $E_0(X) <
E_0(X_0)$, where $X$ is defined as in
Lemma~\ref{lem:genPfrommerAnalysis} with all $c_i^{(j)} = 0$ except
for $c_m^{(m)}$.  So we may assume that $X_0$ does not contain any
$\vec{0}$ columns.  Next, if column $s$ of $X_0$ is an eigenvector
which does not correspond to one of the $m$ most negative eigenvalues,
then an eigenvector $y_t$ ($t > m$) that does correspond to one of the
$m$ most negative eigenvectors is not a column of $X_0$.  Define $X$
as in Lemma \ref{lem:genPfrommerAnalysis} with all $c_i^{(j)} = 0$
except for $c_t^{(s)}$.  Then it is easy to see that there exists $C_2
>0$ such that for $0<|c_t^{(s)}| < C_2$, $E_0(X) < E_0(X_0)$.  This
proves that $X_0$ is a local minimum only if its columns are
eigenvectors corresponding to the $m$ most negative eigenvalues of
$H$.  Therefore, $X_0$ is a global minimum.
\end{proof}

We remark that we can say a little more to classify the critical
points of $E_0$.  It is apparent from \eqref{E0ThroughOrigin} that the
only local maximum of $E_0$ is at the origin.  Therefore, all
non-origin critical points are either saddle points or global minima.

\subsection{OMM energy functional with $\ell^1$ penalty}

The major goal for the rest of the paper is to introduce and analyze a modification of the OMM energy functional which favors sparsity in its solution.  In order to favor sparsity, we add an $\ell^1$ penalty term to the OMM energy functional.  In particular, the new functional we consider is
\begin{equation}
\label{defEmu}
E_\mu(X) = E_0(X) + \mu \tnorm{X}_1,
\end{equation}
where $\tnorm{X}_1 = \sum_i \sum_j |X_{ij}|$ is the entrywise $\ell^1$
norm (viewing $X$ as a vector).  By minimizing \eqref{defEmu}, we hope
to find solutions which are both sparse and ``near'' the global minima
of $E_0$.  In particular, as $\mu \to 0^+$, we would hope that the
minima of $E_\mu$ approach the minima of $E_0$.  In
Lemma~\ref{lem:convwrtmu}, we will show that this is true.
Additionally, in Theorem~\ref{thm:convRateMu}, we quantify the rate at
which the minima of $E_\mu$ approach the minima of $E_0$ as $\mu \to
0^+$.  Before stating these results, we define some notation that we
will use throughout.
\begin{align}
S_\mu &= \argmin E_\mu(X),\\
d(A,B) &= \sup_{a \in A} \inf_{b \in B} \norm{a-b}_F.
\end{align}

Now we prove that the set of minimizers of $E_\mu$ converges to the set of minimizers of $E_0$ as $\mu \to 0^+$.  Even more, $S_\mu$ converges to the elements of $S_0$ which have minimal $\ell^1$ norm.

\begin{lemma}
\label{lem:convwrtmu}
\emph{(Convergence as $\mu \to 0^+$)} Let $M = \argmin\limits_{X \in S_0} \tnorm{X}_1$.  Then, $d(S_\mu, M) \to 0$ as $\mu \to 0^+$.
\end{lemma}
\begin{proof}
First, we note that by the definition of $S_\mu$ ($\mu > 0$), if $X \in S_\mu$ and $Y \in S_0$, then
\begin{equation}
E_0(X) + \mu \tnorm{X}_1 = E_\mu(X) \le E_\mu(Y) = \min E_0 + \mu \tnorm{Y}_1
\end{equation}
This implies that $\tnorm{X}_1 \le \tnorm{Y}_1$.  So, for all $\mu > 0$,
\begin{equation}
\sup_{X \in S_\mu} \tnorm{X}_1 \le \inf_{Y \in S_0} \tnorm{Y}_1. \label{eqBoundednessOfXk}
\end{equation}
Now, let $\{\mu_k\}$ be a sequence that converges down to 0.  Let $X_k \in S_{\mu_k}$ and $Y \in S_0$.  By definition of $S_{\mu_k}$,
\begin{equation}
E_0(X_k) + \mu_k \tnorm{X_k}_1 = E_{\mu_k}(X_k) \le E_{\mu_k}(Y) = \min E_0 + \mu_k \tnorm{Y}_1
\end{equation}
which implies
\begin{equation}
E_0(X_k) - \min E_0 \le \mu_k (\tnorm{Y}_1 - \tnorm{X_k}_1) \le \mu_k \tnorm{Y}_1
\end{equation}
Since this holds for all $k \in \mathbb{N}$ and $\mu_k \to 0$, we
conclude that $E_0(X_k) \to \min E_0$.  Now suppose by way of
contradiction that there exists $\epsilon > 0$ and a subsequence
$\{X_{k_l}\}$ such that $d(X_{k_l},S_0) > \epsilon$ for all $l$.
Eq.~\eqref{eqBoundednessOfXk} implies that there exists a convergent
subsequence $\{X_{{k_l}_m}\}$.  And by continuity of $E_0$, we get our
contradiction: $X_{{k_l}_m} \to X^* \in S_0$.  Since $\mu_k$ and $X_k$
were arbitrary, we have shown that $d(S_\mu,S_0) \to 0$ as $\mu \to
0^+$.  Now, by \eqref{eqBoundednessOfXk}, it is clear that $d(S_\mu,M)
\to 0$ as $\mu \to 0^+$.
\end{proof}

Now that we have proven convergence, we further state and prove the
convergence rate.

\begin{theorem}
\label{thm:convRateMu}
\emph{(Convergence rate as $\mu \to 0^+$)}
Let $\alpha \in (0,1)$.  Suppose $\lambda_m < \lambda_{m+1}$.  Then there exists $\mu_0(\alpha) > 0$ such that for $0< \mu < \mu_0(\alpha)$,
\begin{equation}
d(S_\mu, S_0) \le \frac{\mu \sqrt{Nm}}{\alpha (\lambda_{m+1}-\lambda_m)}.
\end{equation}
\end{theorem}
\begin{proof}
  Without loss of generality, assume that the eigenvectors of $H$
  associated to $\lambda_1, ..., \lambda_N$ are the standard basis vectors, $\vec{e}_1, ...,
  \vec{e}_N$, respectively (a change of basis could always be used to
  bring the problem into this form).  Then, Theorem
  \ref{thm:minimizingSetOfE0} implies that the global minimizers of
  $E_0$ are given by the set
\begin{equation}
\argmin E_0 = \biggl\{ \begin{bmatrix}
G \\
0 \end{bmatrix} \in \mathbb{C}^{N \times m} : G \in U(m) \biggr\}.
\end{equation}
Let $X \in \mathbb{C}^{N \times m}$.  Since the set $\argmin E_0$ is
closed, the set $\argmin \{\norm{X - G}_F : G \in \argmin E_0 \}$ is
nonempty.  By Lemma \ref{lem:invariance}, we can rotate $X$ without
changing $E_{0}(X)$ such that
\begin{equation*}
  \tilde{I} = \begin{bmatrix}
    I_m \\
    0 \end{bmatrix} \in \argmin \bigl\{\norm{X - G}_F: G \in \argmin
  E_0 \bigr\}.
\end{equation*}
Next, we wish to show that for such $X$, we have $X_{ij} = X_{ji}$ for
$i < j \le m$.  For a given $i, j, \theta$, define $G^{(i,j)}(\theta)
\in \mathbb{C}^{N \times m}$ such that
\begin{align}
G^{(i,j)}_{ii}(\theta) &= G^{(i,j)}_{jj}(\theta) = \cos(\theta), \\
G^{(i,j)}_{ij}(\theta) &= -G^{(i,j)}_{ji}(\theta) = \sin(\theta), \\
G^{(i,j)}_{kl}(\theta) &= \tilde{I}_{kl}, \qquad \text{ for } (k,l) \notin \{(i,i), (i,j), (j,i), (j,j)\}.
\end{align}
Note that $G^{(i,j)}(\theta) \in \argmin E_0$ for all $\theta$, and $G^{(i,j)}(0) = \tilde{I}$.  It easy to show that
\begin{align}
\frac{d}{d\theta} \left. \norm{X-G^{(i,j)}(\theta)}_F^2 \right|_{\theta = 0} &= 2 \,\Re(X_{ji} - X_{ij}),\\
\frac{d}{d\theta} \left. \norm{X-iG^{(i,j)}(\theta)}_F^2 \right|_{\theta = 0} &= 2 \,\Im(X_{ji} - X_{ij}).
\end{align}
But these must equal 0 since $\tilde{I} \in \argmin\{\norm{X - G}_F: G
\in \argmin E_0 \}$.  Therefore, $X_{ij} = X_{ji}$.  Next, we show
that $\Im(X_{jj}) = 0$ for $1 \le j \le m$.  Define $G^{(j,j)}$ as
$\tilde{I}$ except with $G^{(j,j)}_{jj} = 1+i\,\Im(X_{jj})$.  It is
clear that if $\Im(X_{jj}) \neq 0$, then $\norm{X-G^{(j,j)}}_F <
\norm{X-\tilde{I}}_F$ which contradicts the assumption that $\tilde{I}
\in \argmin\limits_{G \in \argmin E_0} \norm{X - G}_F$.  Therefore,
$\Im(X_{jj}) = 0$ for all $j$.

Using these facts along with Lemma \ref{lem:genPfrommerAnalysis}, we obtain
\begin{align}
E_0(X) -\min E_0 &= \sum_{i=1}^m \sum_{k=m+1}^N (\lambda_k - \lambda_i)|X_{ik}-\tilde{I}_{ik}|^2 + 4 \sum_{i=1}^m |\lambda_i| \,(\Re(X_{ii}-\tilde{I}_{ii}))^2 \notag\\
  &\qquad + \sum_{i < j} \left(|\lambda_i| + |\lambda_j|\right) \left|(X_{ji}-\tilde{I}_{ji}) + (X_{ij}-\tilde{I}_{ij})\right|^2 + O\left(\norm{X-\tilde{I}}_F^3\right) \label{fullPfrommer} \\
  &\ge (\lambda_{m+1} - \lambda_{m}) \norm{X-\tilde{I}}_F^2 + O\left(\norm{X-\tilde{I}}_F^3\right). 
\end{align}
The condition that $X_{ij} = X_{ji}$ was used to ensure that the $X_{ij} - \tilde{I}_{ij}$ and $X_{ji} - \tilde{I}_{ji}$ terms did not cancel each other out, and condition $\Im(X_{jj}) = 0$ is necessary since the imaginary part of the diagonal entries does not occur in \eqref{fullPfrommer}.  Since the 3rd order terms are bounded near $\tilde{I}$, there exists a $\delta > 0$ such that for $d(X, \tilde{I}) < \delta$,
\begin{equation}
E_0(X) - \min E_0 \ge \alpha (\lambda_{m+1}-\lambda_m) d(X,\tilde{I})^2.
\end{equation}
Since we assumed \emph{without loss of generality} that $\tilde{I} \in \argmin\limits_{G \in \argmin E_0} \norm{X - G}_F$, we have actually proven the stronger statement:  If $d(X, S_0) < \delta$, then 
\begin{equation}
E_0(X) - \min E_0 \ge \alpha (\lambda_{m+1}-\lambda_m) d(X,S_0)^2. \label{convProofLowerBound}
\end{equation}
By Lemma \ref{lem:convwrtmu}, there exists $\mu_0 > 0$ such that $0 \le \mu < \mu_0$ implies $d(S_\mu,S_0) < \delta$.  For such $\mu$, \eqref{convProofLowerBound} is satisfied with $X = X_\mu$ for any $X_\mu \in S_\mu$.
\begin{equation}
E_0(X_\mu) - \min E_0 \ge \alpha (\lambda_{m+1}-\lambda_m) \, d(X_\mu,S_0)^2. \label{quadraticInequal}
\end{equation}

We will now derive one final inequality which will complete the proof.  First we observe that
\begin{equation}
E_0(X_\mu) + \mu \tnorm{X_\mu}_1 = E_\mu(X_\mu) \le E_0(X_0) = E_0(X_0) + \mu \tnorm{X_0}_1,
\end{equation}
for any $X_\mu \in S_\mu$ and $X_0 \in S_0$.  If in particular, we choose $X_0 \in \argmin\limits_{G \in \argmin E_0} \norm{X_\mu - G}_F$, then we can show
\begin{align}
E_0(X_\mu) - \min E_0 &\le \mu \left(\tnorm{X_0}_1 - \tnorm{X_\mu}_1\right) \notag\\
  &\le \mu \tnorm{X_0 - X_\mu}_1 \notag\\
  &\le \mu \tnorm{X_0 - X_\mu}_0 \norm{X_0 - X_\mu}_F \notag\\
  &\le \mu \tnorm{X_0 - X_\mu}_0 \, d(X_\mu, S_0), 
\end{align}
where $\tnorm{Y}_0$ denotes the number of nonzero entries in $Y$.  In particular, we can minimize $\tnorm{X_0 - X_\mu}_0$ to obtain
\begin{equation}
E_0(X_\mu) - \min E_0 \le \mu \, d(X_\mu, S_0) \cdot \min\left\{\tnorm{X_0-X_\mu}_0 : X_0 \in \argmin\limits_{G \in \argmin E_0} \norm{X_\mu - G}_F \right\}. \label{differenceOfE0s}
\end{equation}
We complete the proof by combining \eqref{quadraticInequal} with \eqref{differenceOfE0s} and taking the $\max$ over $X_\mu \in S_\mu$.
\begin{align}
d(S_\mu, S_0) &\le \frac{\mu}{\alpha(\lambda_{m+1}-\lambda_m)} \cdot \max\limits_{X_\mu \in S_\mu} \min \left\{ \tnorm{X_0 - X_\mu}_0 : X_0 \in \argmin\limits_{G \in \argmin E_0} \norm{X_\mu - G}_F \right\} \label{SmutoS0rate} \\
  &\le \frac{\mu \sqrt{Nm}}{\alpha (\lambda_{m+1} - \lambda_m)}. \label{SmutoS0ratePessimistic}
\end{align}
The proof is hence completed. We note that while
\eqref{SmutoS0ratePessimistic} is a simple expression, there may exist
problems for which \eqref{SmutoS0rate} is also a tractable expression
and provides a tighter bound.
\end{proof}

We note that if $X_0 \in S_0$, then $\norm{X_0} = \sqrt{m}$.  So, the relative error does not depend explicitly on $m$.  Theorem \ref{thm:convRateMu} also enables us to prove several corollaries.  As a first application of the theorem, let us prove the convergence rate of $\min E_\mu$ and $E_0(X_\mu)$ to $\min E_0$ (where $X_\mu \in S_\mu$).

\begin{corollary} 
\label{cor:minEmutominE0}
Let $\alpha \in (0,1)$.  For all $\mu > 0$, we have
\begin{equation}
\min E_\mu - \min E_0 \le \mu \min_{X_0 \in S_0} \tnorm{X_0}_1.
\end{equation}
And if $\lambda_m < \lambda_{m+1}$, then there is $\mu_0 > 0$ such that for $0 < \mu < \mu_0$, 
\begin{align}
E_0(X_\mu) - \min E_0 &\le \frac{\mu^2}{\alpha(\lambda_{m+1}-\lambda_m)} \cdot \min \left\{ \tnorm{X_0 - X_\mu}_0^2 : X_0 \in \argmin\limits_{G \in \argmin E_0} \norm{X_\mu - G}_F \right\}, \\
  &\le \frac{\mu^2 Nm}{\alpha (\lambda_{m+1} - \lambda_m)}.
\end{align}
\end{corollary}
\begin{proof}
The first inequality is straightforward by noticing that if $X_\mu \in S_\mu$ and $X_0 \in S_0$, then $E_\mu(X_\mu) \le E_\mu(X_0)$.  The second inequality is obtained by combining \eqref{quadraticInequal} and \eqref{differenceOfE0s}.
\end{proof}

We will perform numerical tests in Section \ref{subsec:numericsconvwrtmu} which suggest that the powers on $\mu$ in these bounds are optimal.  

As a second corollary, we examine how orthonormal the columns of minimizers of $E_\mu$ are.  The minimizers of $E_0$ are known to have orthonormal columns (Theorem \ref{thm:minimizingSetOfE0}), but the addition of the $\ell^1$ term in $E_\mu$ discards this property.  Nonetheless, we can still provide a bound on how close to orthonormal the columns are, as in the next corollary.
\begin{corollary}
\label{cor:orthConv}
Let $\alpha \in (0,1)$.  Then, there exists $\mu_0(\alpha) > 0$ such that for $0 < \mu < \mu_0(\alpha)$ and $X_\mu \in S_\mu$, we have
\begin{equation}
\norm{X_\mu^* X_\mu - I}_F \le \frac{2\mu m \sqrt{N}}{\alpha^2 (\lambda_{m+1} - \lambda_m)}.
\end{equation}
\end{corollary}
\begin{proof}
Choose $\mu$ small enough so that Theorem \ref{thm:convRateMu} applies and $d(S_\mu,S_0) \le 2\sqrt{m}(\frac{1}{\alpha} - 1)$.  Let $X_\mu \in S_\mu$.  Let $X_0 \in \argmin\limits_{X \in S_0} \norm{X_\mu - X}_F$.  Then we can estimate,
\begin{align}
\norm{X_\mu^* X_\mu - I}_F &= \norm{X_\mu^* X_\mu - X_0^* X_0}_F \notag\\
  &\le \norm{X_\mu^* (X_\mu - X_0)}_F + \norm{(X_\mu^* - X_0^*)X_0}_F \notag\\
  &\le \left(\norm{X_\mu}_F + \norm{X_0}_F\right) \norm{X_\mu - X_0}_F \notag\\
  &\le \frac{2\mu m \sqrt{N}}{\alpha^2 (\lambda_{m+1} - \lambda_m)},
\end{align}
where the last line uses Theorem \ref{thm:convRateMu} and the fact that $\norm{X_0}_F = \sqrt{m}$ since the columns of $X_0$ are orthonormal.
\end{proof}

As our final corollary, we examine the convergence rate of the density matrix.  Note that if we choose any $X_0 \in S_0$, then the density matrix $P_{0} = X_0 X_0^*$ will be the projection operator onto the low-lying eigenspace.  In what follows, if $X_\mu \in S_\mu$, then we will consider two different approximations for the density matrix.
\begin{align}
\tilde{P}_{X_\mu} &= X_\mu X_\mu^*,\\
P_{X_\mu} &= X_\mu(X_\mu^*X_\mu)^{-1}X_\mu^*.
\end{align}
$P_{X_\mu}$ is a legitimate projection matrix while $\tilde{P}_{X_\mu}$ is approximately a projection matrix.  The range of each is $\col(X_\mu)$.  We can obtain the following estimates for the convergence of $\tilde{P}_{X_\mu}$ and $P_{X_\mu}$ to $P_0$.

\begin{corollary}
\label{cor:densityConv}
Let $\alpha \in (0,1)$.  Then, there exists $\mu_0(\alpha) > 0$ such that for $0 < \mu < \mu_0(\alpha)$, we have
\begin{align}
\norm{\tilde{P}_{X_\mu} - P_{0}}_F &\le \frac{2\mu m \sqrt{N}}{\alpha^2 (\lambda_{m+1} - \lambda_m)}, \label{convWRTmuPtilde} \\
\norm{P_{X_\mu} - P_{0}}_F &\le \frac{4\mu m \sqrt{N}}{\alpha^2 (\lambda_{m+1} - \lambda_m)}. \label{convWRTmuP}
\end{align}
\end{corollary}
\begin{proof}
Choose $\mu$ small enough so that Theorem \ref{thm:convRateMu} applies and $d(S_\mu,S_0) \le 2\sqrt{m}(\frac{1}{\alpha} - 1)$.  Let $X_\mu \in S_\mu$ and $X_0 \in \argmin_{X \in S_0} \norm{X_\mu - X}_F$.  Then we can estimate,
\begin{align}
\norm{\tilde{P}_{X_\mu} - P_{0}}_F &= \norm{X_\mu X_\mu^* - X_0 X_0^*}_F \notag\\
  &\le \norm{X_\mu (X_\mu^* - X_0^*)}_F + \norm{(X_\mu - X_0)X_0^*}_F \notag\\
  &\le \left(\norm{X_\mu}_F + \norm{X_0}_F\right) \norm{X_\mu - X_0}_F \notag\\
  &\le \frac{2\mu m \sqrt{N}}{\alpha^2 (\lambda_{m+1} - \lambda_m)}.
\end{align}
A slightly more complicated calculation can be used to obtain \eqref{convWRTmuP}.
\begin{align}
\norm{P_{X_\mu} - P_{0}}_F &= \norm{X_\mu (X_\mu^*X_\mu)^{-1} X_\mu^* - X_0 X_0^*}_F \notag\\
  &\le \norm{X_\mu (X_\mu^*X_\mu)^{-1} X_\mu^* - X_\mu X_\mu^*}_F + \norm{\tilde{P}_{X_\mu} - P_0}_F \notag\\
  &= \norm{I - X_\mu^*X_\mu}_F + \norm{\tilde{P}_{X_\mu} - P_0}_F \notag\\
  &\le \frac{4\mu m \sqrt{N}}{\alpha^2 (\lambda_{m+1} - \lambda_m)}.
\end{align}
The last line uses Corollary \ref{cor:orthConv} and \eqref{convWRTmuPtilde}.  We went from the second to the third line by use of
\begin{align}
\norm{X_\mu (X_\mu^* X_\mu)^{-1} X_\mu^* - X_\mu X_\mu^*}_F^2 &= \tr \left[X_\mu \left[(X_\mu^* X_\mu)^{-1} - I\right] X_\mu^* X_\mu \left[(X_\mu^* X_\mu)^{-1} - I\right] X_\mu^* \right] \notag\\
  &= \tr \left[ X_\mu^* X_\mu \left[(X_\mu^* X_\mu)^{-1} - I\right] X_\mu^* X_\mu \left[(X_\mu^* X_\mu)^{-1} - I\right]\right] \notag\\
  &= \tr\left[ (I - X_\mu^*X_\mu) (I - X_\mu^*X_\mu) \right] \notag\\
  &= \norm{I - X_\mu^*X_\mu}_F^2.
\end{align}
\end{proof}

\subsection{Local minima of $E_\mu$}

In Theorem \ref{thm:localglobalminima}, we showed that all local
minima of $E_0$ are global minima.  Unfortunately, this is no longer
true for $E_\mu$.  However, we can provide some information about the
local minima of $E_\mu$.  The next theorem states that local minima of
$E_\mu$ are in a sense ``generated by'' critical points of $E_0$.
Additionally, it states that these local minima of $E_\mu$ can be
forced to be arbitrarily close to the critical points of $E_0$ by
choosing $\mu$ sufficiently small.  We denote the set of local minima
of $E_\mu$ by $S_\mu^\text{loc}$ and denote the set of critical points
of $E_0$ by $C_p$.

\begin{theorem}
\label{thm:localMinGenByCPs}
For every $\epsilon > 0$, there exists $\delta > 0$ such that $0 \le \mu < \delta$ implies $d(S_\mu^\text{loc}, C_p) < \epsilon$.
\end{theorem}
\begin{proof}
First we need a preliminary result:  Let $X_\text{min} \in S_\mu^\text{loc}$ for some $\mu > 0$.  Let $\{x_i\}$ be the coordinate directions corresponding to the nonzero entries of $X_\text{min}$ and let $\{z_j\}$ be coordinate directions corresponding to the zero entries of $X_\text{min}$.  Then we must have
\begin{align*}
\frac{\partial E_0}{\partial x_i} (X_\text{min}) &\in \{\pm \mu\} \\
\left|\frac{\partial E_0}{\partial z_j} (X_\text{min})\right| &\le \mu
\end{align*}
for all $i,j$.  This implies $\norm{\nabla E_0 (X_\text{min})}_F \le \mu \sqrt{Nm}$.  

Next, let $U \in \mathbb{C}^{N\times m}$ such that $\norm{U}_F = 1$.  $E_0$ along the direction of $U$ is given by 
\begin{align}
E_0(tU) &= -\tr[U^*UU^*HU] t^4 + 2 \tr[U^*HU] t^2 \notag\\
  &= c_1(U) t^4 + c_2(U) t^2. \label{E0ThroughOrigin}
\end{align}
Since $H$ is negative definite, $c_2(U) < 0 < c_1(U)$.  So, on any line through the origin, $E_0$ has a double well potential shape.  All critical points of $E_0$ must either be at the origin or at one of the minima along some line through the origin.  By taking the derivative with respect to $t$, we see that if $Y \in C_p$, then $\norm{Y}_F = 0$ or $\norm{Y}_F = \sqrt{\frac{-c_2(Y/\norm{Y}_F)}{c_1(Y/\norm{Y}_F)}}$.  Note that this expression is continuous with respect to $Y$ for $Y \ne 0$.  So, we can apply a compactness argument on the projective space to see that the set $C_p$ is bounded.

Now to prove the lemma, we suppose by way of contradiction that the lemma is false.  Then, there exists $\epsilon > 0$ and sequences $\{\mu_k\}$ and $\{X_k\}$ such that $\mu_k \to 0$ and $X_k \in S_{\mu_k}^\text{loc}$ and $d(X_k,C_p) \ge \epsilon$.  First note that any local minima $Y$ of $E_\mu$ must have a Frobenious norm less than $\sqrt{\frac{-c_2(Y/\norm{Y}_F)}{c_1(Y/\norm{Y}_F)}}$.  This is clear because
\begin{equation}
E_\mu(tU) = c_1(U) t^4 + c_2(U) t^2 + \tnorm{U}_1 \mu |t|.  \label{EmuThroughOrigin}
\end{equation}
Therefore, the set $\{X_k\}$ is bounded.  So now by compactness, there exists a convergent subsequence $X_{k_l} \to Z$.  Also, note that $\norm{\nabla E_0(X_k)}_F \le \mu_k \sqrt{Nm}$.  Since $\mu_k \to 0$ and the gradient of $E_0$ is continuous, we have $\norm{\nabla E_0(Z)}_F = 0$.  That is, $Z \in C_p$.  Therefore, there exists $L$ such that $d(X_{k_L},C_p) < \epsilon$ which contradicts our assumption.
\end{proof}

It is also worth commenting on the existence of local minima of $E_\mu$.  We note that the origin is a local minima of $E_\mu$ for all $\mu > 0$ (as can be seen from \eqref{EmuThroughOrigin}).  We also note that local minima can be generated for all $\mu > 0$ at points other than the origin.  As an example, if $H = \bigl[ \begin{smallmatrix} 
-1 & 0 \\
0 & -2 \end{smallmatrix} \bigr]$ and $X \in \mathbb{C}^{2 \times 1}$, then the critical points $[e^{\alpha i}, 0]^T$ for all $\alpha \in \R$ generate a local minima of $E_\mu$ for all $\mu > 0$.  However, for small $\mu$, the basins around the local minima are small, so algorithms are less likely to converge to those points.  

\section{Algorithms}

Having determined many nice analytic properties of the functional
$E_\mu$, we now seek a practical way to find its minima using
numerical optimization algorithms.  We have two main objectives for
our algorithms.  First, we want them to find a minimum quickly.
Second, we want them to maintain sparsity from iteration to iteration.
The hope is that sparsity can be used to speed up the computational
time per iteration through the use of sparse matrix arithmetic.  We
choose the Iterative Shrinkage--Thresholding Algorithm (ISTA)
\cite{beck2009fast} as our base algorithm to which we subsequently
make modifications.  This algorithm is a natural choice because it is
designed to deal with objective functions which are the sum of a
differentiable function ($E_0$) and a convex function (the $\ell^1$
penalty term).

We will present a total of six different algorithms for minimizing
$E_\mu$.  All of the algorithms are essentially modifications of ISTA.
They can be divided into two main categories: block and non--block.
The block algorithms update $X$ column by column instead of updating
all entries of $X$ at once.  Additionally, our algorithms could be
categorized based on whether they use traditional backtracking or a
more aggressive dynamic backtracking we propose.  The numerical
results in Section \ref{subsec:compareConvOfAlgs} show that the
algorithms with this dynamic backtracking converge much faster than
the algorithms using traditional backtracking.

\subsection{Non--block version}

ISTA with backtracking \cite{beck2009fast} is given in algorithm \ref{alg:ISTA}.  In the algorithm, $T_\alpha$ is the shrinkage operator and $\circ$ is an entrywise product.  More specifically, 
\begin{align}
(T_\alpha Y)_{ij} &= \left\{ \begin{array}{cl}
         0 & \mbox{if $|Y_{ij}|\le \alpha$},\\
        (|Y_{ij}| - \alpha) \, e^{i \arg(Y_{ij})} & \mbox{if $|Y_{ij}|>\alpha$}.\end{array} \right.\\
A \circ B &= \sum_i \sum_j A_{ij} B_{ij}.
\end{align}
where $\arg(z)$ denotes the complex argument of $z$.  The algorithm is essentially a steepest descent (corresponding to
$E_0$) algorithm combined with an entrywise soft thresholding. The
backtracking is used to choose the step size, which will be further
explained below.

\begin{algorithm}[H]
\caption{ISTA with backtracking}
\label{alg:ISTA}
\begin{algorithmic}[1]
\Require $X_0 \in \mathbb{R}^{N \times m}, \, L_0 > 0, \, \eta > 1$
\For{$k = 1,2,...$}
\Comment{Loop until convergence is achieved}
\State $L_k^{(1)} \leftarrow L_{k-1}$
\ForUntilBegin{$j = 1,2,...$}

\Comment{Backtracking iteration}
	\State $X_k^{(j)} \leftarrow T_{\mu/L_k^{(j)}} \left(X_{k-1} - \frac{1}{L_k^{(j)}} \nabla E_0(X_{k-1}) \right)$
	\State $L_k^{(j+1)} \leftarrow \eta L_k^{(j)}$
\ForUntilEnd{$E_0(X_k^{(j)}) \le E_0(X_{k-1}) + \nabla E_0(X_{k-1}) \circ (X_k^{(j)}-X_{k-1}) + \frac{L_k^{(j)}}{2} \norm{X_k^{(j)}-X_{k-1}}_F^2$}
\State $X_k \leftarrow X_k^{(j)}$
\State $L_k \leftarrow L_k^{(j)}$
\EndFor
\end{algorithmic}
\end{algorithm}

The condition $\norm{X_k - X_{k-1}}_F < \mathrm{tol}$ can be easily
used as a convergence criteria as the quantity $\norm{X_k -
  X_{k-1}}_F$ needs to be calculated anyway.

\subsubsection{Dynamic backtracking}

Numerical results indicated that the convergence could be quite slow
using the backtracking version of ISTA.  So we improved the algorithm
by choosing $L_k$ in a more dynamic fashion.  At the start of each
iteration, we choose the value of $L_k^{(1)}$ based on the previous
iterations.  Backtracking is still used if necessary.

To see the idea behind this modification, it is necessary to recall the derivation of ISTA.  In particular, the basic idea is to replace the objective function $E_\mu$ with an approximation $\tilde{E}_\mu$ of the objective function near $X_{k-1}$,
\begin{equation}
\tilde{E}_\mu(X_{k-1};X) = E_0(X_{k-1}) + \nabla E_0(X_{k-1}) \circ (X-X_{k-1}) + \frac{L_k}{2} \norm{X-X_{k-1}}_F^2 + \mu \tnorm{X}_1.
\end{equation}
Then each step of ISTA is just given by $X_k = \argmin \tilde{E}_\mu(X_{k-1};X)$.  So, we can see that if $L_k$ is a good approximation to the second directional derivative in the direction $X - X_{k-1}$, then $\tilde{E}_\mu$ would be a good approximation of $E_\mu$ in that direction.  For this reason, we choose $L_k^{(1)}$ by approximating the second directional derivative.  In particular, when $k \ge 2$, we make the following modification to line 2 of algorithm \ref{alg:ISTA}.
\smallskip
\begin{algorithmic}[1]
\makeatletter\setcounter{ALG@line}{1}\makeatother
\State $\displaystyle L_k^{(1)} \leftarrow c_1 \, \frac{\norm{\nabla E_0(X_{k-1}) - \nabla E_0(X_{k-2})}_F}{\norm{X_{k-1}-X_{k-2}}_F}$
\end{algorithmic}
\smallskip 
where $c_1 > 1$ (we found that $c_1 = 1.5$ worked well).  The constant $c_1$ is used to prevent our guess for $L_k^{(1)}$ from being too small (which would lead to backtracking).  This guess for $L_k^{(1)}$ will typically be much smaller than if we used the traditional backtracking approach $L_k^{(1)} = L_{k-1}$.  This allows us to take larger step sizes which leads to faster convergence.

The drawbacks of this method are the extra computational cost
associated with calculating $L_k^{(1)}$ and the extra cost of having
to do more backtracking steps.  However, our numerical results show
that these extra costs are worthwhile because the accelerated
convergence leads to many fewer iterations being required.

Furthermore, we will now provide another modification to the algorithm
which greatly reduces the number of backtracks that are required.
We can reduce the number of backtracks by modifying line 5 in the algorithm.  If line 2 happens to give a guess for $L_k$ which is way too small, then multiple backtracks will be required in a single iteration.  To reduce the occurrence of multiple backtracks, we can choose $L_k^{(j+1)}$ in a smarter way.  In particular, we can replace line 5 by
\smallskip
\begin{algorithmic}[1]
\makeatletter\setcounter{ALG@line}{4}\makeatother
\State $\displaystyle L_k^{(j+1)} \leftarrow c_2 \, \frac{2\left[ E_0(X_k^{(j)}) - E_0(X_{k-1}) - \nabla E_0(X_{k-1}) \circ (X_k^{(j)}-X_{k-1}) \right]}{\lVert X_{k}^{(j)}-X_{k-1}\rVert_F^2}$
\end{algorithmic}
\smallskip
where $c_2 > 1$ (we used $c_2 = 2$).  Notice that this calculation is essentially ``free'' since all these terms need to be calculated anyway for the break condition in line 6.  The new line 5 is obtained by simply solving for $L_k^{(j)}$ in the break condition (and then multiplying by a constant $c_2$ greater than 1 to try to avoid multiple backtracks per iteration).  In practice, using $c_2 > 1.1$ seemed to virtually eliminate the occurrence of multiple backtracks per iteration.

\subsection{Block version}
\label{subsec:blockAlg}

\begin{algorithm}[h]
\caption{Block version with dynamic backtracking}
\label{alg:blockISTA}
\begin{algorithmic}[1]
\Require $X_0 \in \mathbb{R}^{N \times m}, \, L_1^{(1)} > 0, \, c_1, c_2 > 1$
\For{$k = 1,2,...$}
\Comment{Loop until convergence is achieved}
	\State Choose $b_k \in \{1, 2, ..., m\}$.
	\State $\displaystyle{L_k^{(1)} \leftarrow c_1 \, \frac{\norm{\nabla_{b_k} E_0(X_{k-1}) - \nabla_{b_k} E_0(X_{p(b_k)-1})}_F}{\norm{X_{k-1,b_k}-X_{p(b_k)-1,b_k}}_F}} \qquad $ (if $k \ge 2$)
	\ForUntilBegin{j = 1,2,...}
	
	\Comment{Backtracking iteration}
		\State $X_{k,b_k}^{(j)} \leftarrow T_{\mu/L_k^{(j)}} \left(X_{k-1,b_k} - \dfrac{1}{L_k^{(j)}} \nabla_{b_k} E_0(X_{k-1}) \right)$
		\State $\displaystyle{L_k^{(j+1)} \leftarrow c_2 \, \frac{2\left[ E_0(X_k^{(j)}) - E_0(X_{k-1}) - \nabla_{b_k} E_0(X_{k-1}) \circ (X_{k,b_k}^{(j)}-X_{k-1,b_k}) \right]}{\norm{X_{k,b_k}^{(j)}-X_{k-1,b_k}}_F^2}}$
	\ForUntilEnd{$E_0(X_k^{(j)}) \le E_0(X_{k-1}) + \nabla_{b_k} E_0(X_{k-1}) \circ (X_{k,b_k}^{(j)}-X_{k-1,b_k}) + \dfrac{L_k^{(j)}}{2} \norm{X_{k,b_k}^{(j)}-X_{k-1,b_k}}_F^2$}
\State $X_k \leftarrow X_{k-1}$
\State $X_{k,b_k} \leftarrow X_{k,b_k}^{(j)}$
\Comment{Only the $b_k$ block changes}
\EndFor
\end{algorithmic}
\end{algorithm}

We now present our block algorithm.  This algorithm (with dynamic backtracking) is given in algorithm \ref{alg:blockISTA}.  In order to implement it, we must first choose how to split up $X$ into blocks.  We choose each column of $X$ to be its own block.  This gives a total of $m$ blocks.  By choosing the blocks this way, $m$ iterations of the block algorithm has essentially the same cost as one iteration of the non-block algorithm.

The following notation is used in algorithm \ref{alg:blockISTA}.  $X_{\ell,b_k}$ is the $b_k$ block of $X_\ell$, and $\nabla_{b_k}$ is the gradient with respect to the $b_k$ block, holding all other blocks constant.  We also make use of the function 
\begin{equation}
p(b_k) = \left\{
     \begin{array}{cl}
       \max \{\ell < k : b_\ell = b_k\} & \text{ if } \{\ell < k : b_\ell = b_k\} \ne \emptyset, \\
       0 & \text{ if } \{\ell < k : b_\ell = b_k\} = \emptyset.
     \end{array}
   \right.
\end{equation}
That is, $p(b_k)$ is the most recent iteration that the $b_k$ block was updated, and if the $b_k$ block has not yet been updated, then we define $p(b_k) = 0$.

There are clearly many possible ways to choose $b_k$ in line 2 of each iteration, but for our numerical tests in Section \ref{sec:numerics}, we will only consider two such strategies.  In our ``sequential'' strategy, we simply loop sequentially through the blocks ($1,2,...,m,1,2...$).  In what we call the ``random'' strategy for choosing $b_k$, we choose a random permutation of $\{1,...,m\}$ and then use that for the first $m$ iterations.  For the next $m$ iterations, we choose a new random permutation of $\{1,...,m\}$, and so on.  In this way, each block is updated exactly once between iterations $(q-1)m+1$ and $qm$ for any $q \in \mathbb{N}$.  With our choices of $b_k$, we note that $m$ iterations of the block algorithm is approximately the same as one iteration of the non--block algorithm (since each block is updated exactly once).  So, when we compare the two methods, we think of $m$ iterations of the block algorithm as being a single iteration.

\section{Numerical results}
\label{sec:numerics}

In this section we perform numerical tests to verify our theoretical
results (Section \ref{subsec:numericsconvwrtmu}), compare the
performance of our proposed algorithms (Section
\ref{subsec:compareConvOfAlgs}), and explore some of the more
practical concerns of our method (Sections \ref{subsec:localMinTest}, \ref{subsec:numericsICs} and \ref{subsec:dynMu}).  We will use the following example,
borrowed from \cite{gao2009orbital}, as our test problem in the
remainder of the paper.
\begin{align}
H &= -\frac{1}{2}\frac{d^2}{dx^2} + V(x) - \eta, \\
V(x) &= \alpha \sum_j e^{-(x-r_j)^2/(2\beta^2)},
\end{align}
where $\eta$ is a constant chosen so that $H$ is negative definite.
The domain of the problem is $[0,10]$ with periodic boundary
conditions, and $\{r_j\}_{j=1}^{10} = \{0.5,1.5,...,9.5\}$ represent
the atom positions.  We discretize on a uniform spacial grid and use
the centered difference formula to discretize the second derivative.
We take $m = 10$, and the initial condition $X_0$ is set as follows.
The $i^\text{th}$ column of $X_0$ has support of width $2L+1$ centered
at $r_i$.  On the support, the initial condition is set by generating
random numbers between 0 and $2/(2L+1)$ so that the expected 2--norm
of each column is 1.  Unless otherwise specified, we will use the
parameters $\alpha = -100$ and $\beta = 0.1$.  In this case, the
spectral gap is $\lambda_{m+1} - \lambda_m \approx 54.2$.

\subsection{Convergence with respect to $\mu$}
\label{subsec:numericsconvwrtmu}

The first test we perform is to confirm the theoretical results in
Theorem \ref{thm:convRateMu} and Corollary \ref{cor:minEmutominE0}.  We
list the minimal value of $E_\mu$ (and the value of $E_0$ at the
minimum) compared to the minimal value of $E_0$ to validate Corollary
\ref{cor:minEmutominE0}.  We also investigate the distance from a
minimizer of $E_\mu$ to the set $S_0$ to test the linear convergence
rate guaranteed by Theorem \ref{thm:convRateMu}.

To perform this test, we take $N = 800$, and use two different choices for $\alpha$ which give us different spectral gaps.  In particular, we choose $\alpha = -100$ and $\alpha = -10$ which give gaps of 54.2 and 4.36 respectively.

The results for different values of $\mu$ are tabulated in Tables \ref{tab:convwrtmuLargeGap} and \ref{tab:convwrtmuSmallGap}.  From these results, we see that the powers of $\mu$ in our theoretical results appear to be optimal.  We also see that the problem with a smaller gap converges slower than the one with a larger gap in accordance with the constant out front in Theorem \ref{thm:convRateMu} and Corollary \ref{cor:minEmutominE0}.

To get an intuitive view of what is happening to the solution as we change $\mu$, we have plotted the 5th column of $X$ in Figure \ref{fig:sparsityWRTmuAndGap}.  From this figure, we can see the localizing behavior of the orbital as $\mu$ gets larger.  We also see that for small $\mu$, the orbital is not much different than when $\mu = 0$.

\begin{table}
\centering
\begin{tabular}{ c|cc|cc|cc }
  $\mu$ & $\min E_\mu - \min E_0$ & order & $E_0(X_\mu) - \min E_0$ & order & $d(X_\mu, S_0)$ & order \\
  \hline
  $2^{-8}$ & 2.4412e-01 & -- & 7.5520e-05 & -- & 1.1147e-03 & -- \\
  $2^{-9}$ & 1.2208e-01 & 0.99976 & 2.1369e-05 & 1.8213 & 5.9890e-04 & 0.89628 \\
  $2^{-10}$ & 6.1045e-02 & 0.99987 & 5.8575e-06 & 1.8672 & 3.1342e-04 & 0.93420 \\
  $2^{-11}$ & 3.0524e-02 & 0.99993 & 1.6183e-06 & 1.8558 & 1.6449e-04 & 0.93012 \\
  $2^{-12}$ & 1.5262e-02 & 0.99996 & 4.4340e-07 & 1.8678 & 8.5349e-05 & 0.94653 \\
\end{tabular}
\caption{Large gap problem:  Convergence with respect to $\mu$.}
\label{tab:convwrtmuLargeGap}
\end{table}

\begin{table}
\centering
\begin{tabular}{ c|cc|cc|cc }
  $\mu$ & $\min E_\mu - \min E_0$ & order & $E_0(X_\mu) - \min E_0$ & order & $d(X_\mu, S_0)$ & order \\
  \hline
  $2^{-8}$ & 4.5203e-01 & -- & 2.6667e-03 & -- & 2.0772e-02 & -- \\
  $2^{-9}$ & 2.2668e-01 & 0.99576 & 6.9602e-04 & 1.9853 & 1.0505e-02 & 0.9835 \\
  $2^{-10}$ & 1.1351e-01 & 0.99783 & 1.9228e-04 & 1.9431 & 5.4162e-03 & 0.9556 \\
  $2^{-11}$ & 5.6799e-02 & 0.99891 & 5.3083e-05 & 1.8623 & 2.7386e-03 & 0.9838 \\
  $2^{-12}$ & 2.8410e-02 & 0.99946 & 1.3839e-05 & 1.9264 & 1.2868e-03 & 1.0896 \\
\end{tabular}
\caption{Small gap problem:  Convergence with respect to $\mu$.}
\label{tab:convwrtmuSmallGap}
\end{table}

\begin{figure}
\centering
\begin{subfigure}{.5\textwidth}
  \centering
  \includegraphics[scale=.6]{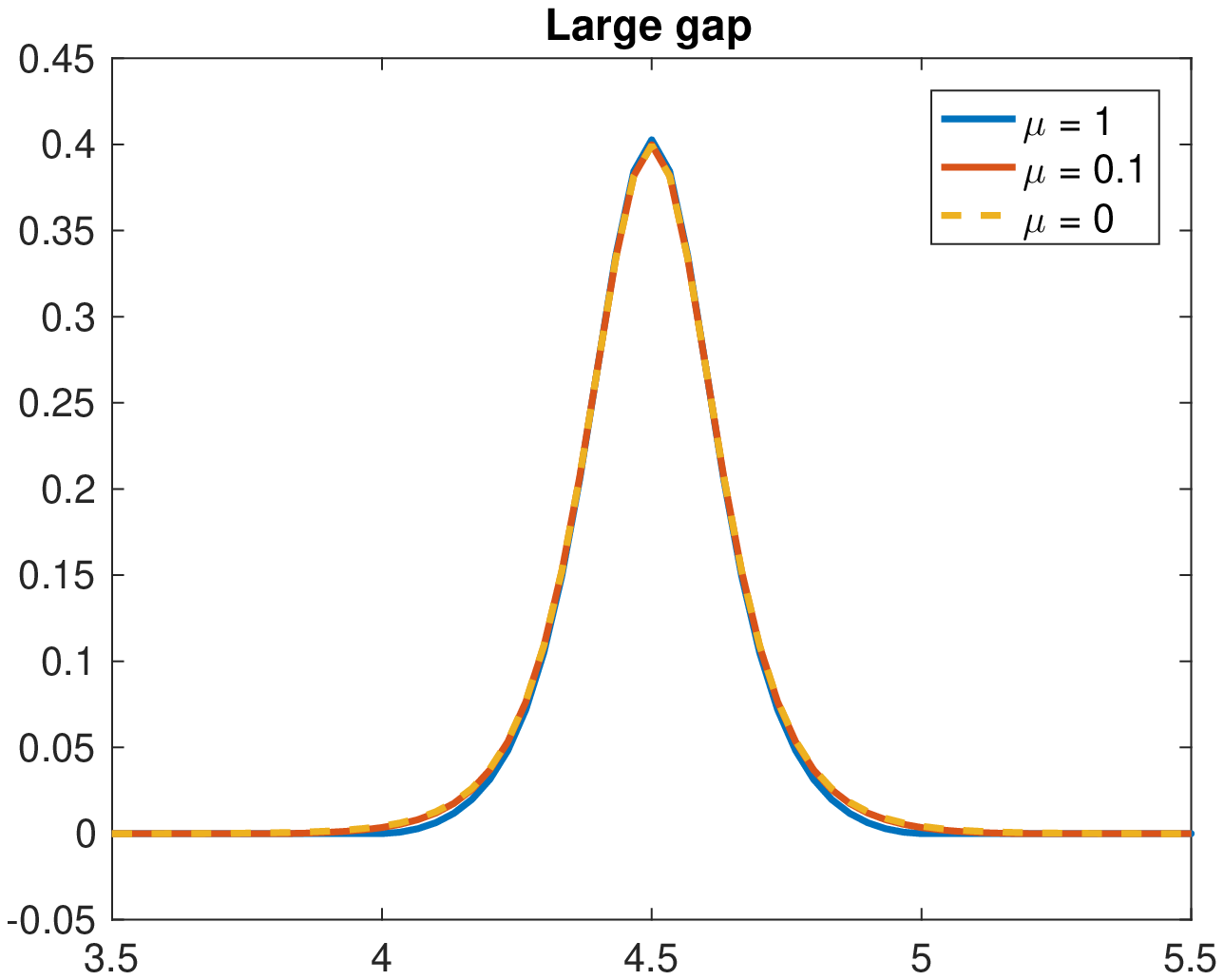}
\end{subfigure}%
\begin{subfigure}{.5\textwidth}
  \centering
  \includegraphics[scale=.6]{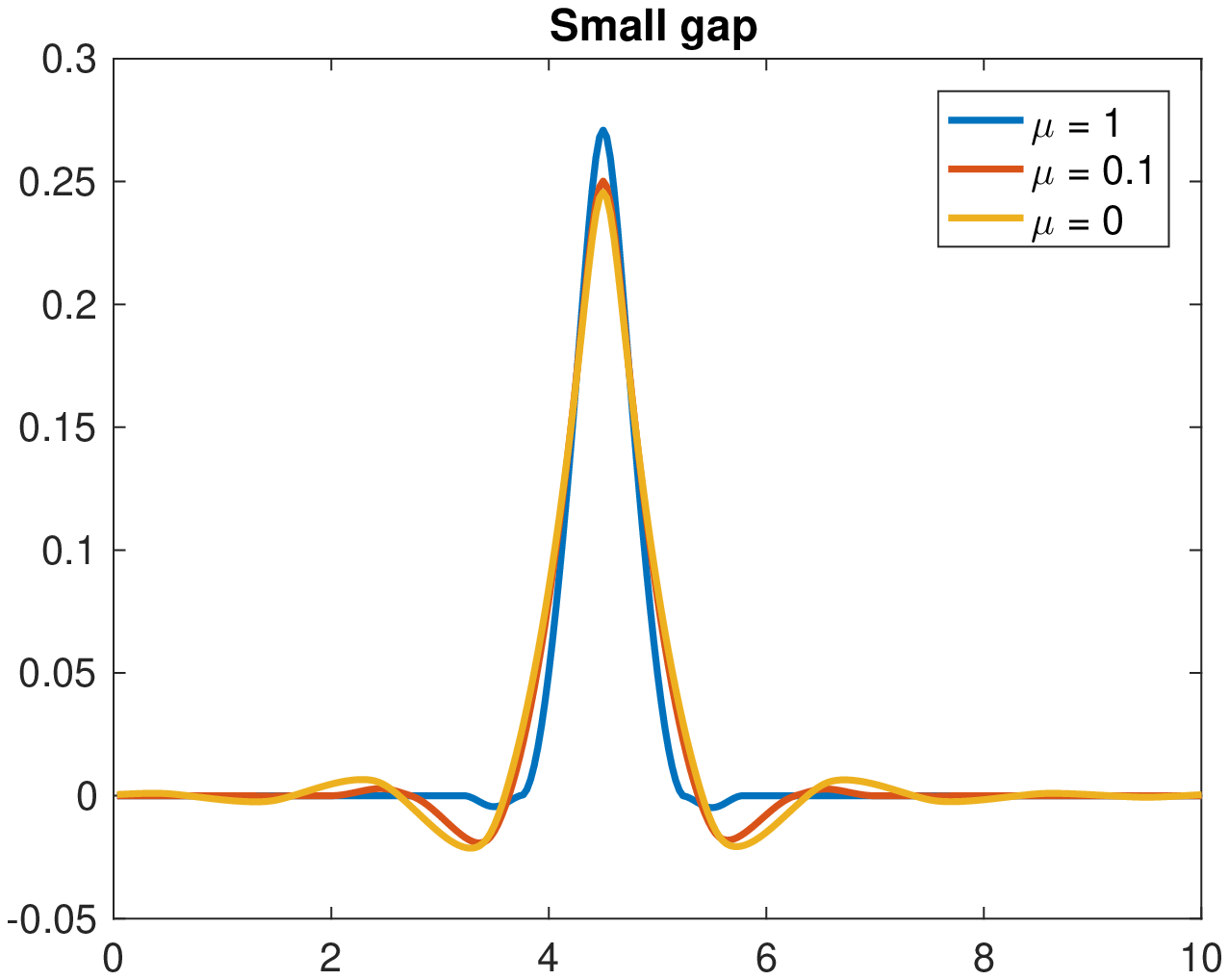}
\end{subfigure}
\caption{Here the 5th column of $X$ is plotted for different values of $\mu$.  The result from the large gap problem is zoomed in for a better view.}
\label{fig:sparsityWRTmuAndGap}
\end{figure}

\subsection{Comparing convergence of algorithms}
\label{subsec:compareConvOfAlgs}

Here we test the convergence rates of our proposed algorithms.  The
results in Figure \ref{fig:convRateTest} are typical when the initial
conditions have small support, \textit{i.e.}, are better
approximations of the solution, since the solution decays away from
the $r_i$'s.  The dependence on initial conditions will be discussed
further in Section \ref{subsec:numericsICs}.  We can see that the
dynamic backtracking algorithms converge much faster than the
traditional backtracking algorithms.  This increase in speed is due to
the fact that the dynamic algorithms can use much larger step sizes
than the traditional algorithms.  The convergence rate appears to be
linear, and there does not appear to be much of a difference between
the required number of iterations for the block and non--block
versions.

\begin{figure}
\centerline{\includegraphics[scale=.6]{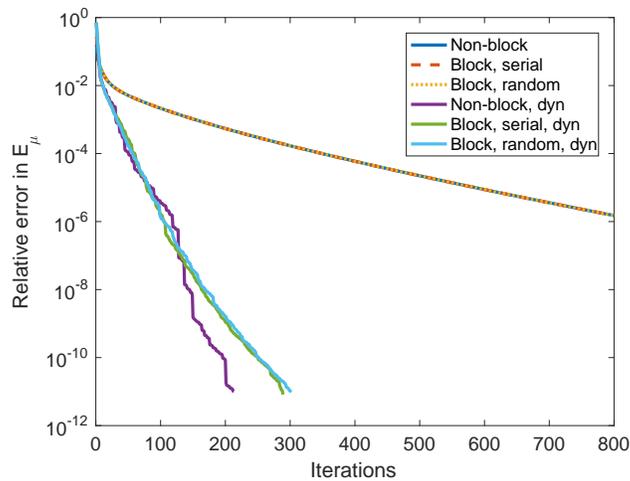}}
\caption{Convergence rate test}
\label{fig:convRateTest}
\end{figure}

\subsection{Local minima test}
\label{subsec:localMinTest}

Our method is not the first to try and take advantage of the ``almost compact'' structure of the orbitals.  A previous attempt was the truncated steepest descent (truncated SD) method.  However, it is known to have the unfortunate property of converging to local minima \cite{gao2009orbital}.  Since our functional $E_\mu$ is also known to contain local minima, we wish to test how often our proposed method gets trapped in local minima compared to the truncated SD method.  Our method has an additional parameter $\mu$ whose effect also needs to be tested.  It is intuitive that as $\mu$ gets larger, it creates larger ``basins'' around the local minima of $E_\mu$.  So we expect our algorithms to get stuck at local minima more often when $\mu$ is larger.

The following test compares many things:  truncated SD vs our proposed methods, block vs non--block, and the effect of using different values of $\mu$.  Our test setup is as follows.  For a given initial condition, we run the truncated SD method as well as the block and non--block dynamic backtracking methods with different values for $\mu$.  We use the large gap problem with $N = 500$ and $L = 60$.  All these methods have different global minima of their respective energy functionals, so to compare them all we plot (Figure \ref{fig:localMinTest}) how far above this minimum the respective algorithms converged to.  That is, for the $\mu = 10$ test, we plot $E_{10}(X) - \min E_{10}$.  We do likewise for the other values of $\mu$.  For the truncated SD method, we plot $E_0(X) - \min E_0$.

We can see from these results that even though our method is known to sometimes have local minima of $E_\mu$, it doesn't seem to be an issue in our test problem for small $\mu$.  Our method avoids local minima much better than truncated SD for small $\mu$, but gets stuck in local minima much more often for large $\mu$.  We also note that neither the block nor the non--block version of the algorithm seems to be much better or worse at avoiding local minima.  We also ran the same test with 10,000 trials for the non--block version with $\mu = 0.5$, and the algorithm did not get stuck at a local minima even once.

In our results, we notice that when the algorithm gets stuck at local minima, these local minima are essentially grouped into discrete energy levels.  This is especially seen well in the non--block version.  This behavior is expected because the local minima of $E_\mu$ are known to occur near critical points of $E_0$ (Theorem \ref{thm:localMinGenByCPs}) which are spans of eigenvectors of $H$ (Lemma \ref{lem:charOfCPs}).  Therefore, we expect that the energy value at local minima will be approximately the sum of some $m$ or fewer eigenvalues of $H$.  This is what leads to the observed discrete energy levels in the plots.

\begin{figure}[h]
\centering
\begin{subfigure}{.5\textwidth}
  \centering
  \includegraphics[scale=.4]{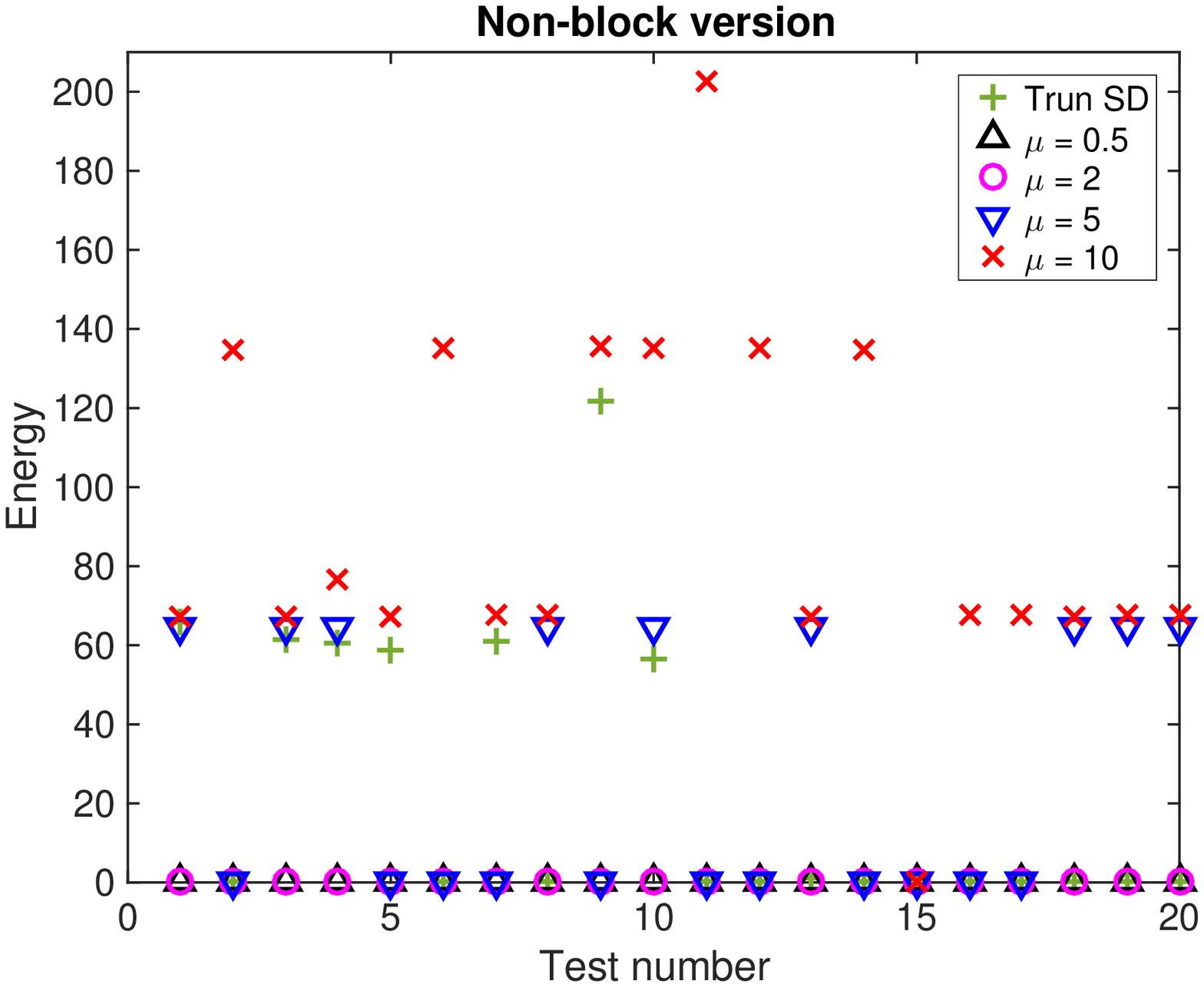}
\end{subfigure}%
\begin{subfigure}{.5\textwidth}
  \centering
  \includegraphics[scale=.4]{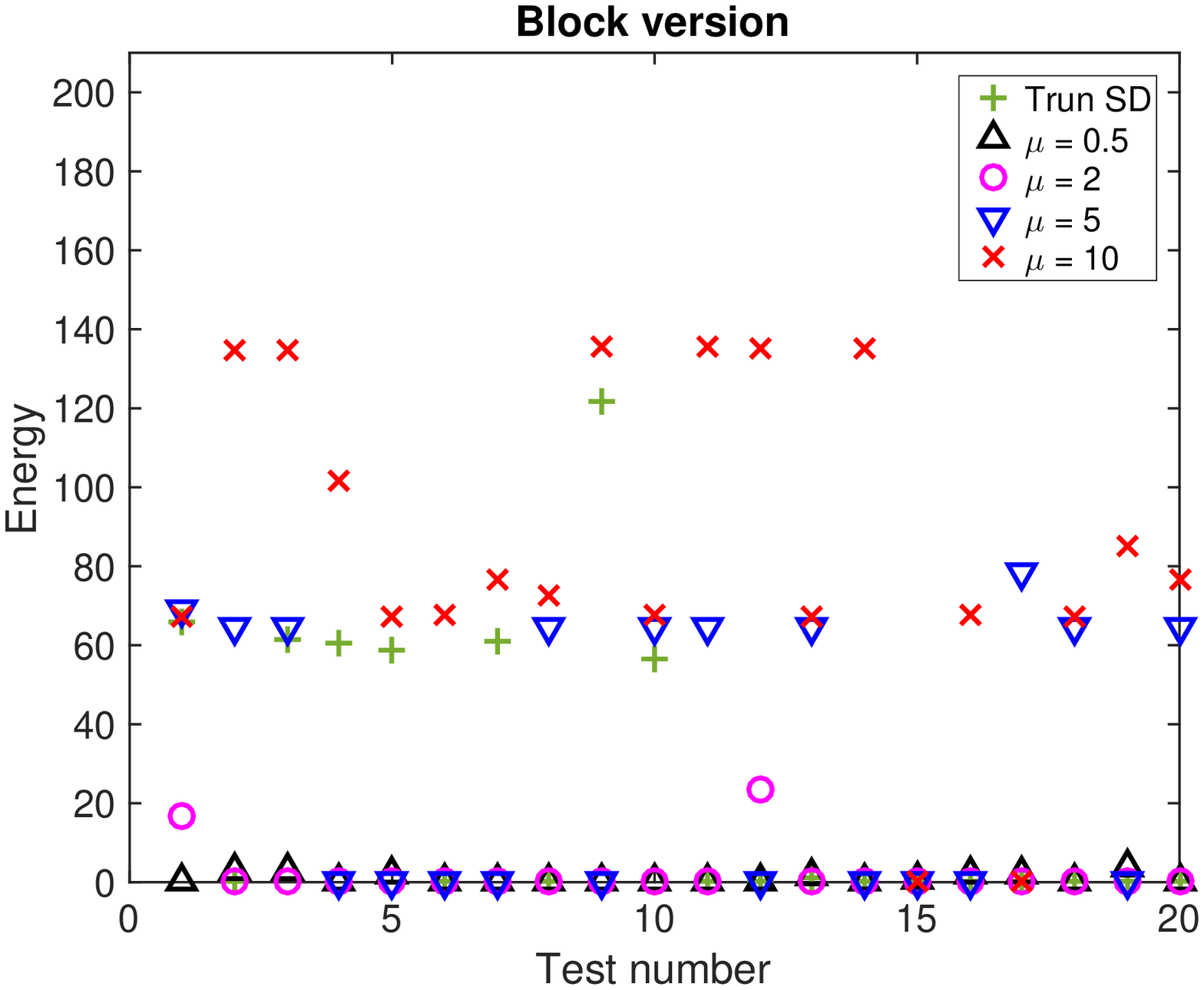}
\end{subfigure}
\caption{Plots of the energy (above the minimum) that the algorithms converged to.  Each column corresponds to a distinct initial condition and the initial conditions are the same in both plots.}
\label{fig:localMinTest}
\end{figure}

\subsection{Dependence on initial conditions}
\label{subsec:numericsICs}

In our next numerical test we investigate how the number of iterations required to converge depends on the initial condition.  In general, it would be expected that a ``good'' initial guess would lead to a smaller number of iterations required for convergence and a ``bad'' initial guess would lead to a larger number of iterations.  We will see that this is true, but there is also more to the story.  While we want to converge in as few iterations as possible, we also want to preserve sparsity from iteration to iteration.  By having a sparse $X$, we hope that the calculations can be done faster using sparse matrix multiplication algorithms.  We will see that this idea of preserving sparsity is actually intimately related to the number of iterations required to converge.

In order to get a better idea of what is going on, we note that sparsity from iteration to iteration is essentially determined in the following step of the algorithm.
\smallskip
\begin{algorithmic}
\State $X_k^{(j)} \leftarrow T_{\mu/L_k^{(j)}} \left(X_{k-1} - \frac{1}{L_k^{(j)}} \nabla E_0(X_{k-1}) \right)$
\end{algorithmic}
\smallskip
The part in the parentheses tends to increase the number of nonzero entries while the shrinkage operator tends to decrease them.  It is clear just from this formula that a larger $\mu$ will tend to preserve sparsity better than a smaller $\mu$.  However, there are two critical reasons for wishing to avoid using a large $\mu$.  The first is that the minima of $E_\mu$ will be farther away from the minima of $E_0$.  The second (as seen in Section \ref{subsec:localMinTest}) is that the basins around the local minima of $E_\mu$ will be larger which increases the chance of converging to a local minima of $E_\mu$.

We will use the dynamic backtracking version of ISTA for this test.  In Figure \ref{fig:sparsityPlots}, we plot the number of iterations that each entry of $X$ was nonzero.  In particular, our test problem uses the parameters $N = 150$, $m = 10$, and $\mu = 0.1$.  The only difference between the tests is the support of the initial conditions.  The support of the initial condition acts as a surrogate for how ``good'' the initial condition is.  With the way we are generating the initial conditions, a smaller support would tend to be closer in shape to the actual solution, and therefore a ``better'' initial guess.  We see from Figure \ref{fig:sparsityConv} that the smaller the value of $L$, the fewer iterations that are required for convergence.

For $L > 4$, we can see plateaus in the plots of Figure \ref{fig:sparsityConv}.  It is worthwhile to explore what is going on here since these plateaus dramatically increase the number of iterations required to converge.  We notice that for larger $L$, the algorithm actually converges to approximately the minimal value of $E_0$ very quickly, but then struggles to decrease its $\ell^1$ norm.  During the plateau region, the algorithm is slowly decreasing its $\ell^1$ norm while keeping its $E_0$ value fairly constant.  In other words, by the beginning of the plateau, the algorithm has found a point very near $S_0$, but this point has a large $\ell^1$ norm.  So, the plateau is essentially spent rotating (recall that $E_0$ is invariant under rotations) the solution into regions of lower $\ell^1$ norm that are also near $S_0$.

To be more concrete, in our tests we have found that these plateau regions are caused by overlap of the orbitals (the $\ell^1$ norm will be smaller when they are non--overlapping).  We can see this in Figure \ref{fig:sparsityPlots}.  For example, in the $L=12$ case, orbitals 6 and 7 overlapped for much of the simulation.  It is as if a Givens rotation had been applied to the 6th and 7th orbitals of an exact minimizer of $E_\mu$.  When two (or more) orbitals are ``competing'' with each other like this, convergence tends to be very slow until the competition is resolved (ie. one of them wins by essentially zeroing out the other one).  This is what causes those long plateaus in the graph of error vs iterations when $L = 8, 12, 16$.  So, it is important to avoid overlap because it leads to both less sparsity (more time per iteration) and slower convergence (more iterations). 

\begin{figure}[h]
\centerline{\includegraphics[scale=.9]{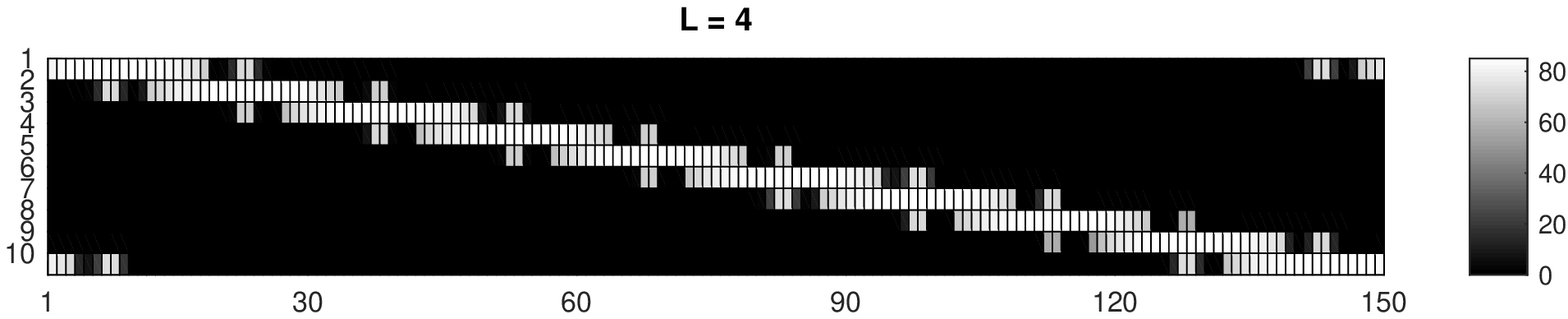}}

\centerline{\includegraphics[scale=.9]{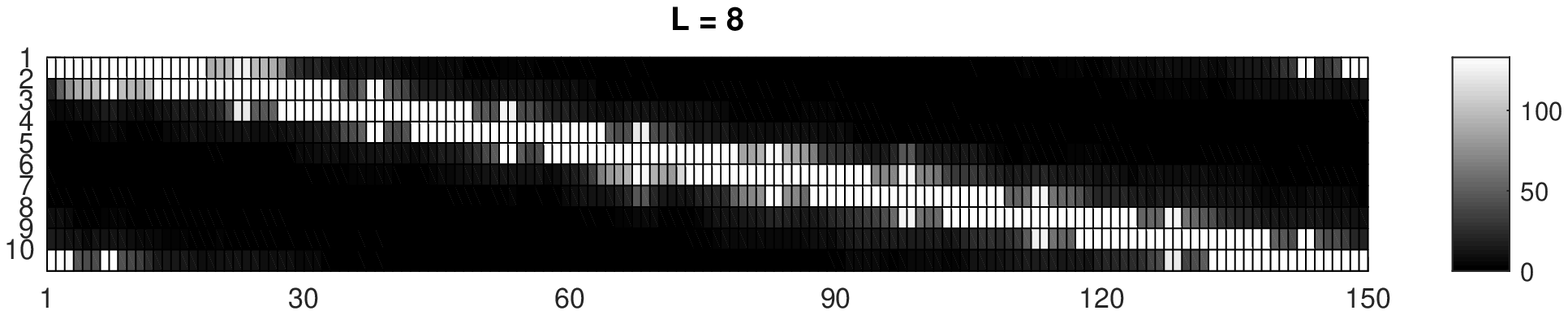}}

\centerline{\includegraphics[scale=.9]{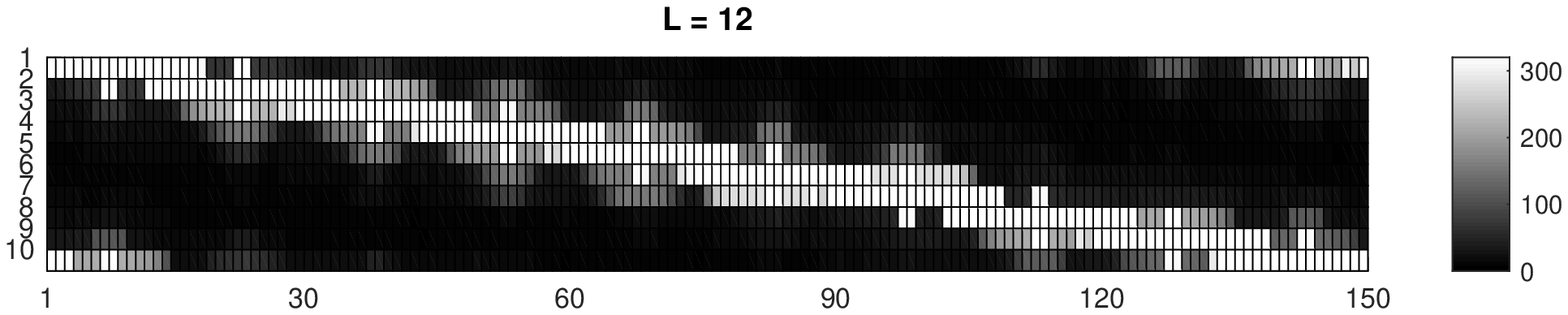}}

\centerline{\includegraphics[scale=.9]{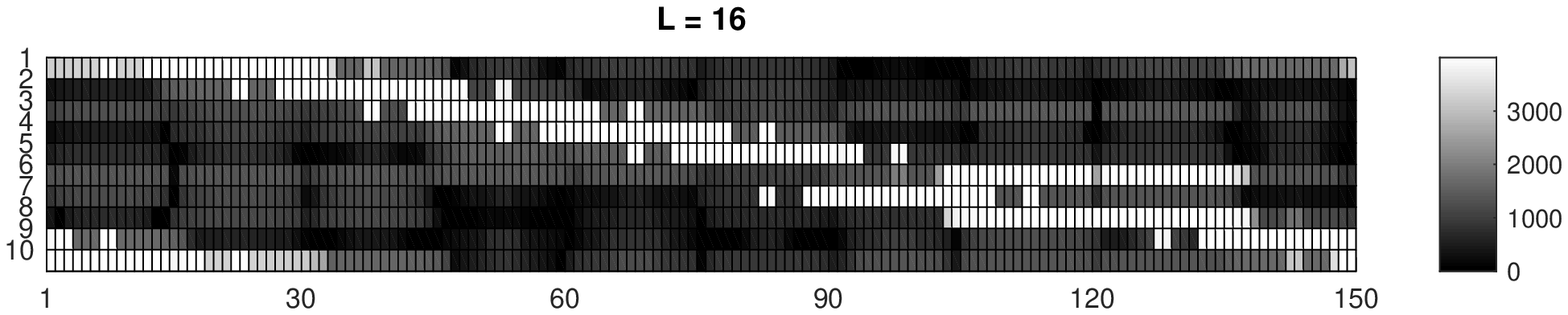}}

\caption{The above plots show the number of iterations for which each particular entry of $X$ was nonzero.  The layout of the plots corresponds to the transpose of $X$.  Note the different scales for each plot.}
\label{fig:sparsityPlots}
\end{figure}

\begin{figure}[h]
\centering
\begin{subfigure}{.5\textwidth}
  \centering
  \includegraphics[scale=.5]{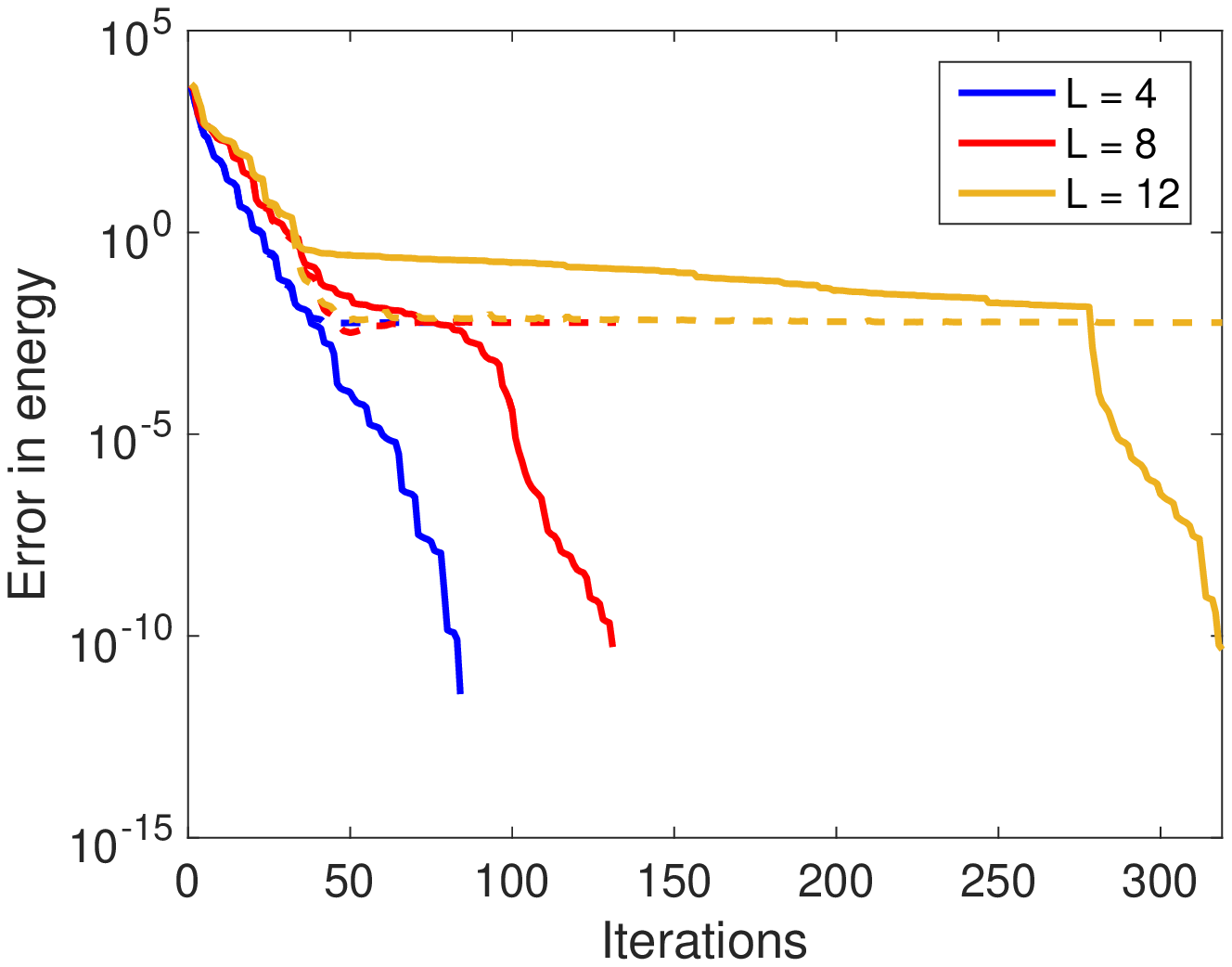}
  \caption{$L = 4,8,12$}
  \label{fig:sub1}
\end{subfigure}%
\begin{subfigure}{.5\textwidth}
  \centering
  \includegraphics[scale=.5]{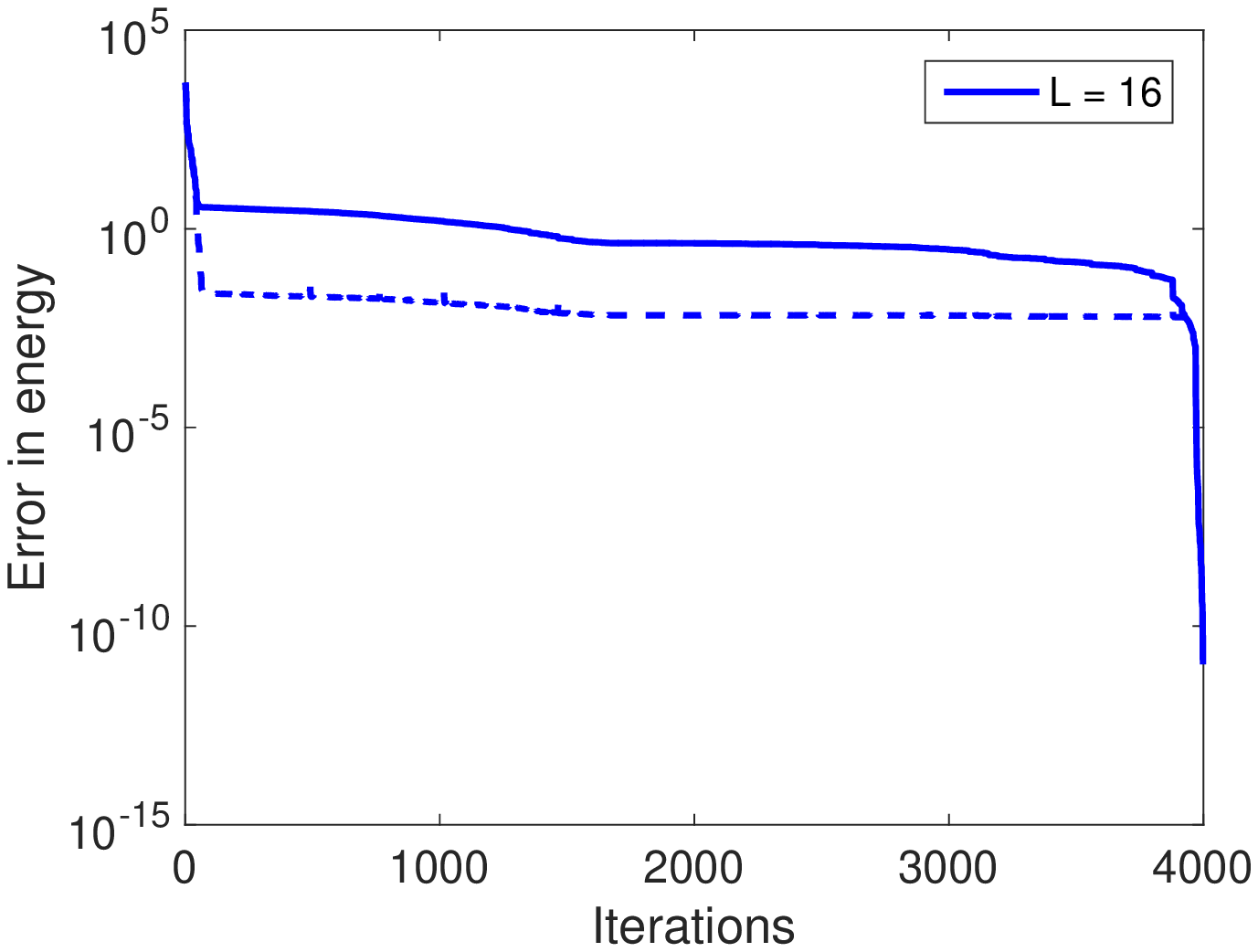}
  \caption{$L=16$}
  \label{fig:sub2}
\end{subfigure}
\caption{Convergence of the dynamic backtracking version of ISTA for different $L$.  The plots here correspond to the sparsity plots in Figure \ref{fig:sparsityPlots}.  Solid lines plot $E_\mu(X) - \min E_\mu$.  Dashed lines plot $E_0(X) - \min E_0$.  Note that we do not expect the error in $E_0$ to go to 0 since that would only happen in the case $\mu = 0$.}
\label{fig:sparsityConv}
\end{figure}

\subsection{Dynamic $\mu$}
\label{subsec:dynMu}

The above discussions on local minima, Theorem \ref{thm:convRateMu}, and maintaining sparsity lead us to some competing objectives for our choice of $\mu$.  On the one hand, we want a small $\mu$ so that the minima of $E_\mu$ are close to the minima of $E_0$ and so that we can avoid being trapped at local minima.  On the other hand, we want a large enough $\mu$ so that we can maintain sparsity which leads to faster convergence.  In this section, our goal is to show the possible benefit of changing $\mu$ from iteration to iteration.  We do not present a robust algorithm, but rather just a numerical test which provides hope that future work in this area could be beneficial.

In general, the main idea for choosing $\mu$ at each iteration should be based on the following points.  First, $\mu$ should be small enough at the beginning of the simulation so that the algorithm has a very low probability of getting stuck at a local minimum.  Second, after the algorithm has found a near minimal value of $E_0$, the value of $\mu$ could be increased in order to facilitate reduction of the $\ell^1$ norm.  A large $\mu$ provides a stronger ``force'' pushing $X$ towards a lower $\ell^1$ norm.  Note that using a larger $\mu$ will also push $X$ towards larger values of $E_0$.  However, as long as $\mu$ is not increased too much, we expect the value of $E_0$ to remain below $\min E_0 + \lambda_{m+1}-\lambda_m$.  Staying below this value would be preferred as we know from Theorem \ref{thm:localMinGenByCPs} and Lemma \ref{lem:charOfCPs} that for small $\mu$, local minima are near values of $E_0$ that correspond to the sum of some $m$ eigenvalues of $H$.  So, we would hope not to get stuck at a local minimum if we only moderately increase $\mu$ once we are already near a minimal value of $E_0$.  Finally, if we want our solution to be close to the exact solution, we want the value of $\mu$ to be small.  So, after the $\ell^1$ norm has been sufficiently decreased, $\mu$ could also be decreased to allow for a solution closer to an exact minimizer of $E_0$.

Of course, such a strategy is not so straightforward as one does not usually know a priori the values of $\min E_0$ and $\min E_\mu$.  So, we leave it as an open problem how one could implement such ideas in practice.  Nevertheless, we perform a numerical test which justifies the idea of increasing $\mu$ after several iterations.  To illustrate the effect that changing $\mu$ in this manner can have, we will carry out the above problem with $L = 16$ again.  We run the algorithm twice (starting with the same initial condition).  One run will use a constant $\mu$ value of $0.1$.  The other will also use a $\mu$ value of 0.1 except for iterations 100--499 during which $\mu$ will be set to 1.  We can see from Figure \ref{fig:changingMu} that the algorithm converges much faster when we change $\mu$ in this way ($\approx 550$ iterations instead of $\approx 2900$).  It is an open problem how this type of strategy could be carried out in a more dynamic and intelligent manner.

\begin{figure}[h]
\centering
\begin{subfigure}{.5\textwidth}
  \centering
  \includegraphics[scale=.5]{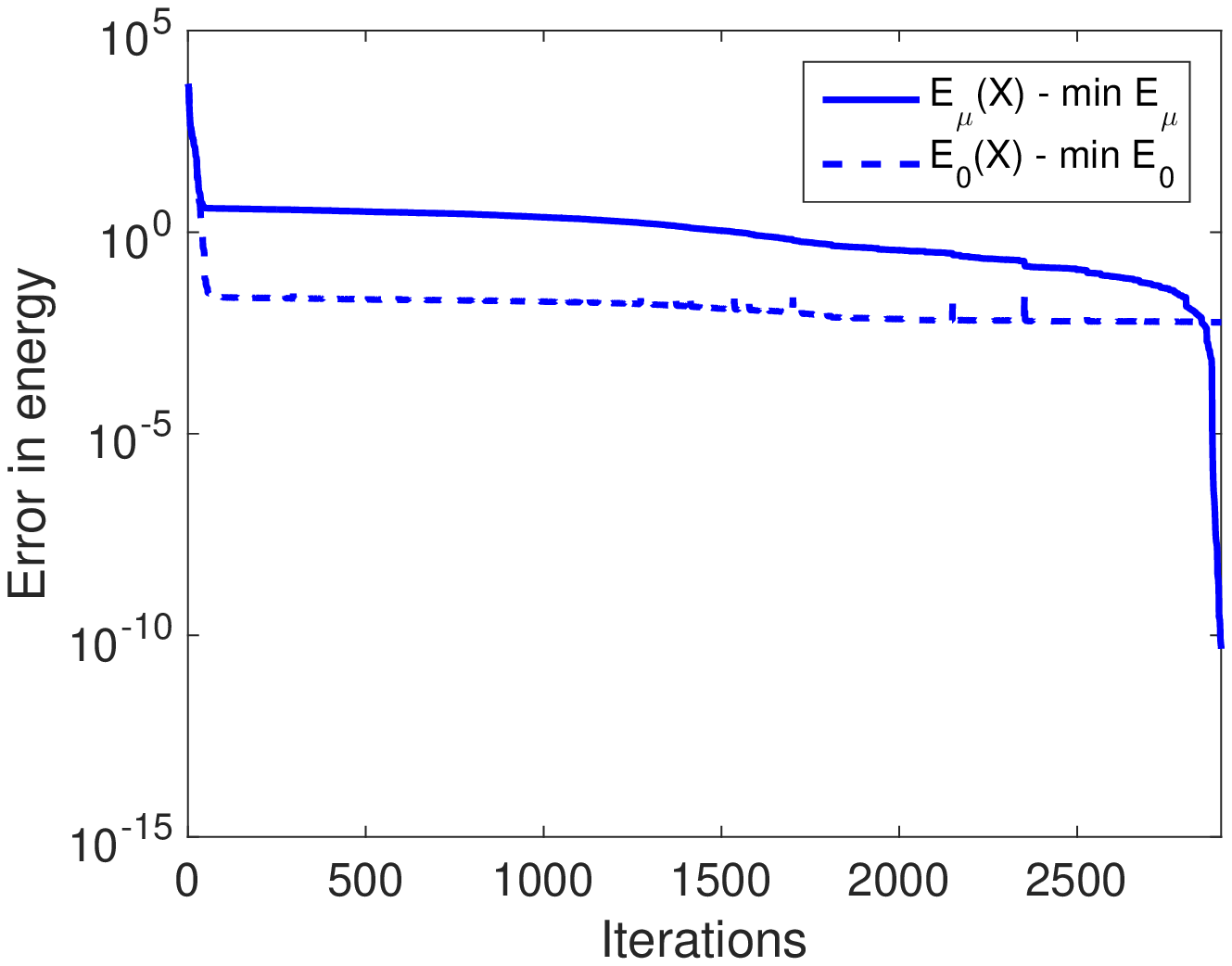}
  \caption{$L = 16$ with constant $\mu$}
  \label{fig:constMu}
\end{subfigure}%
\begin{subfigure}{.5\textwidth}
  \centering
  \includegraphics[scale=.5]{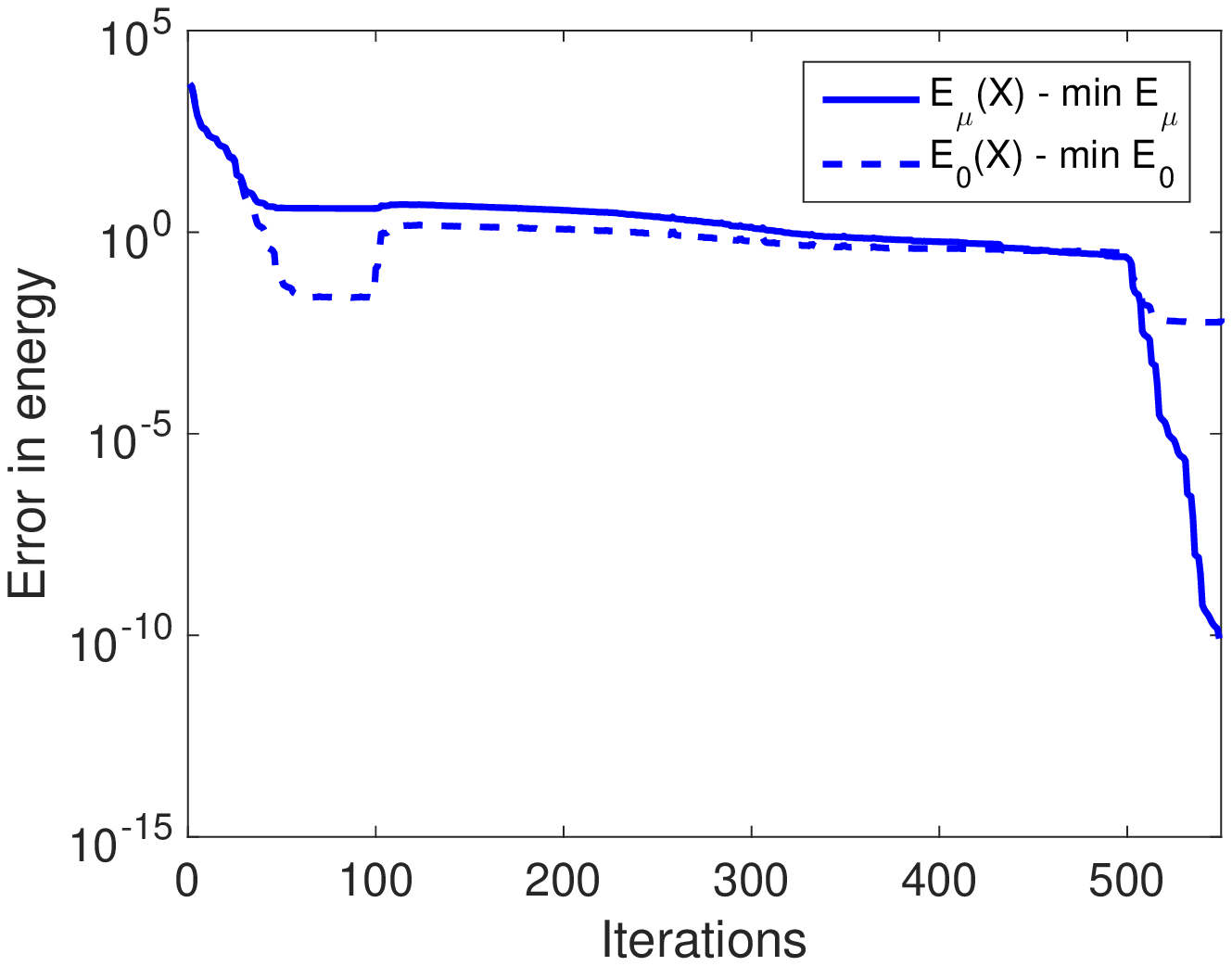}
  \caption{$L=16$ with variable $\mu$}
  \label{fig:varMu}
\end{subfigure}
\caption{Convergence with constant $\mu$ versus a variable $\mu$.  Note the different scales in the two plots.  Also note that for $\mu = 0.1$, $E_0(X_\mu) \approx \min E_0 + 5.9 \cdot 10^{-3}$.  In plot \ref{fig:varMu}, we see the value of $E_0(X)$ jump up when $\mu$ increases and jump down when $\mu$ decreases.  This is not surprising since a larger $\mu$ would take the solution farther away from the minimum value of $E_0$.}
\label{fig:changingMu}
\end{figure}

\section{Discussion}

We have presented a new method for OMM calculations which takes advantage of almost compact orbitals by introducing an $\ell^1$ penalty term into the energy functional.  Our analysis proves the convergence of the minimizers of $E_\mu$ to the minimizers of $E_0$ which justifies the practice of minimizing $E_\mu$ rather than $E_0$.  Numerical results show the ability of our algorithms to maintain sparsity from iteration to iteration.  This suggests that our algorithms combined with sparse matrix algorithms could have high performance.  Implementing such a scheme and comparing with other schemes which minimize the OMM energy functional is a direction for future work.

There are several additional possibilities for future work.  One is to
prove convergence results for our proposed algorithms.  Analysis has
been done on algorithms similar to the ones we present with
traditional backtracking \cite{xu2014globally, tseng2009coordinate}.
It will be interesting to investigate the convergence of the dynamic
backtracking algorithms.  Another possible future direction would be
to parallelize the block algorithm: all the blocks would be updated at
the same time (in parallel) rather than one after another.  The
convergence of such an algorithm needs further investigation.  A
further direction would be to explore the use of a dynamic $\mu$, as
already discussed in Section \ref{subsec:dynMu}.  Last but not
least, the convergence of the algorithm might be accelerated by using
a well chosen preconditioner, for example as considered in the recent work \cite{LuYang}.

\bibliographystyle{abbrv} 
\bibliography{references}

\end{document}